\documentclass[a4paper,11pt]{amsart}
\usepackage[utf8]{inputenc}
\usepackage{lmodern}
\usepackage{latexsym}
\usepackage{a4wide}
\usepackage{amscd}
\usepackage{graphics}
\usepackage{amsmath}
\usepackage{amsthm}
\usepackage{amssymb}
\usepackage{bm}
\usepackage{mathrsfs}
\usepackage{enumerate}
\usepackage{pgf,tikz}
\usepackage{dsfont}
\usepackage{soul}
\usepackage{csquotes}

\usepackage{hyperref}
\hypersetup{
	linktocpage=true,
	colorlinks=true,
	linkcolor=blue!80!black,
	bookmarks=true,
	citecolor=orange!50!red,
	urlcolor=magenta!80!black,
}

\usepackage{cleveref}

\numberwithin{equation}{section}

\newtheorem{theorem}{Theorem}[section]

\newtheorem{lemma}[theorem]{Lemma}

\theoremstyle{definition}
\newtheorem{definition}[theorem]{Definition}
\newtheorem{example}[theorem]{Example}

\newtheorem{remark}[theorem]{Remark}

\newcommand{\R}{\mathbb{R}}	
\newcommand{\N}{\mathbb{N}} 
\newcommand{\dx}{\,\mathrm{d}x}	
\newcommand{\dxdy}{\,\mathrm{d}x \mathrm{d}y}
\newcommand{\dxdt}{\,\mathrm{d}x \mathrm{d}t}
\renewcommand{\d}{\,\mathrm{d}}
\newcommand{\norm}[1]{\left\lVert #1 \right\lVert}
\newcommand{\abs}[1]{\left| #1 \right|}

\newcommand{\sub}{\subseteq}

\DeclareMathOperator{\supp}{\mathrm{supp}}

\DeclareMathOperator{\dive}{\mathrm{div}}

\newenvironment{bvp}{\left\{\begin{aligned}  }{\end{aligned}\right.}

\newcommand{\Ds}{\mathcal{D}^{s,2}(\R^N)} 
\newcommand{\Dext}{\mathcal{D}^{1,2}_t(\overline{\R^{N+1}_+})} 

\newcommand{\weak}{\rightharpoonup}
 
\renewcommand{\>}{\right\rangle}
\newcommand{\LA}{\bm{\lambda},\bm{a}}

\DeclareMathOperator{\Tr}{\mathrm{Tr}} 

\title{On fractional critical problems with multi-polar Hardy potentials}

\usepackage{amsaddr}

\author[E. Mainini, D. Mukherjee, R. Ognibene]{%
	Edoardo Mainini\textsuperscript{1},
	Debangana Mukherjee\textsuperscript{2},
	Roberto Ognibene\textsuperscript{3}%
}

\address{\textsuperscript{1}Dipartimento di Ingegneria meccanica, 
	energetica, gestionale e dei trasporti \\
	Universit\`a degli studi di Genova \\
	Via all'Opera Pia, 15, 16145 Genova, Italy}
\email{edoardo.mainini@unige.it} 

\address{\textsuperscript{2}Department of Mathematics, Krea University \\
	School of Interwoven Arts and Sciences (SIAS) \\
	 Sricity, Andhra Pradesh-517646, India}
\email{debangana.mukherjee@krea.edu.in}

\address{\textsuperscript{3}Dipartimento di Matematica e Applicazioni \\
	Universit\`a degli studi di Milano-Bicocca \\
	Via Roberto Cozzi, 55, 20126 Milano, Italy}
\email{roberto.ognibene@unimib.it}

\begin{document}
	

	\begin{abstract}
		We investigate the existence of positive solutions to fractional equations presenting a double criticality: a multi-polar Hardy-type potential and a Sobolev critical nonlinearity. The nonlocal nature of the operator and the absence of explicit ground states for the single-pole equation stand as major difficulties. We overcome these obstacles by passing to an extended formulation of the problem and by establishing sharp asymptotic estimates for the solutions in the case of a single pole. Then, through a concentration-compactness argument, we show that the existence of minimizers is dictated by the magnitude of the masses and the mutual distances between the corresponding poles.
	\end{abstract}
	
\maketitle

\section{Introduction}
Let $s\in(0,1)$ and $N>2s$. Let us consider $k\geq 1$ real numbers $\bm{\lambda}:=(\lambda_1,\dots,\lambda_k)\in\R^k$ such that $\lambda_1\leq \cdots\leq \lambda_k$  and $k$ points $\bm{a}:=(a_1,\dots,a_k) \in \R^{kN}$ such that $a_i\neq a_j$ for all $i\neq j$. In the present paper, we consider the following linear operator
\begin{equation}\label{eq:frac_oper}
	\mathcal{L}_{\LA}:=(-\Delta)^s-H_{\LA}\qquad\text{in }\R^N,
\end{equation}
which is a Schr\"odinger operator governed by the fractional Laplacian of order $s$
and with a multi-polar Hardy potential
\begin{equation}\label{eq:H_LA}
	H_{\LA}(x):=\sum_{i=1}^k\frac{\lambda_i}{\abs{x-a_i}^{2s}},
\end{equation}
and we are interested in finding nontrivial solutions to the following:
\begin{equation}\label{eq:problem}
		\begin{bvp}
			\mathcal{L}_{\LA}u&=u^{2^*_s-1}, &&\text{in }\R^{N}, \\
			u&\geq 0, &&\text{in }\R^N,
		\end{bvp}
\end{equation}
where $2^*_s=\frac{2N}{N-2s}$ is the fractional Sobolev critical exponent.
In order to rigorously state the object of investigation and the main results, we first need to introduce the functional setting and some notation. 
We define the space $\Ds$ as the completion of $C_c^\infty(\R^N)$ with respect to the norm
\begin{equation*}
	\norm{u}_{\Ds}:=\left(\frac{C_{N,s}}{2}\int_{\R^{2N}}\frac{\abs{u(x)-u(y)}^2}{\abs{x-y}^{N+2s}}\dxdy\right)^{1/2},
\end{equation*}
which is induced by the following scalar product
\begin{equation*}
	(u,v)_{\Ds}:=\frac{C_{N,s}}{2}\int_{\R^{2N}}\frac{(u(x)-u(y))(v(x)-v(y))}{\abs{x-y}^{N+2s}}\dxdy,
\end{equation*}
with
\[
	C_{N,s}:=\frac{2^{2s}}{\pi^{\frac{N}{2}}}\frac{\Gamma\left(\frac{N+2s}{2}\right)}{\Gamma(2-s)}s(1-s)>0.
\]
In fact, we have that the fractional Laplacian acts as follows
\begin{align*}
	&(-\Delta)^s\colon \mathcal{D}^{s,2}(\R^N)\to ((\mathcal{D}^{s,2}(\R^N))^* \\
	&_{(\mathcal{D}^{s,2}(\R^N))^*}\langle(-\Delta)^su,v\rangle_{\mathcal{D}^{s,2}(\R^N)}=(u,v)_{\mathcal{D}^{s,2}(\R^N)}\quad\text{for all }u,v\in\mathcal{D}^{s,2}(\R^N).
\end{align*}
We also denote by
\begin{align*}
	Q_{\LA}(u,v):=(u,v)_{\Ds}-\int_{\R^N}H_{\LA}uv\dx,\qquad
	Q_{\LA}(u):=Q_{\LA}(u,u),
\end{align*}
for all $u,v\in\Ds$. With this in mind, we say that $u\in L^1(\R^N;(1+|x|^{N+2s})^{-1}\dx)$ is a \emph{weak solution} of \eqref{eq:problem} if
\begin{equation}\label{eq:weak}
\begin{bvp}
		&u\in\Ds \\
		&u\geq 0\quad\text{a.e. in }\R^N \\
		&Q_{\LA}(u,v)=\int_{\R^N}u^{2^*_s-1}v\dx\quad\text{for all }v\in\Ds.
\end{bvp}
\end{equation}
In the present paper, we seek for solutions of \eqref{eq:problem} which are (multiple of) minimizers of the corresponding Rayleigh quotient, i.e. which achieve
\begin{equation}\label{eq:min_S_LA}
	S_{\LA}:=\inf_{u\in\Ds\setminus\{0\}}\frac{Q_{\LA}(u)}{\norm{u}_{L^{2^*_s}(\R^N)}^2}.
\end{equation}
Indeed, one can see that if $u\in\Ds$ is a minimum point attaining $S_{\LA}$ and satisfying $\norm{u}_{L^{2^*_s}(\R^N)}=1$, then $(S_{\LA})^{\frac{1}{2-2^*_s}}\,u$ solves the corresponding Euler-Lagrange equation \eqref{eq:problem}, in the sense of \eqref{eq:weak}.
In particular, the main goal of the present paper is to find sufficient conditions on the \emph{masses} $\{\lambda_i\}_{i=1,\dots,k}$ and on the \emph{poles} $\{a_i\}_{i=1,\dots,k}$ which guarantee  existence (or non-existence) of minimizers for \eqref{eq:min_S_LA}. Due to the presence of the critical Hardy-type potential, it turns out that a critical threshold in both cases (existence or non-existence) is played by the best constant $h_{N,s}>0$ in the fractional Hardy inequality
\begin{equation*}
	h_{N,s}\int_{\R^N}\frac{u^2}{|x|^{2s}}\dx\leq \norm{u}_{\Ds}^2\quad\text{for all }u\in \Ds,
\end{equation*}
proved in \cite{Herbst1977}. Moreover, also the following coercivity parameter
\begin{equation}\label{eq:mu}
	\mu_{\LA}:=\inf_{u\in\Ds\setminus\{0\}}\frac{Q_{\LA}(u)}{\norm{u}_{\Ds}^2}
\end{equation}
plays a crucial role, being $0$ a critical threshold in the existence or non-existence of solutions. Finally, we need the following last ingredient in order to state our first main result. Let us consider the problem with a single pole and with $\lambda\in[0,h_{N,s})$:
\begin{equation}\label{eq:one_pole_pbm_i}
	\begin{bvp}
		(-\Delta)^s u-\frac{\lambda}{|x|^{2s}}&=u^{2^*_s-1},&&\text{in }\R^N \\
		u&>0,&&\text{in }\R^N\setminus\{0\}.
	\end{bvp}
\end{equation}
By \cite[Theorem 1.1]{Fall2014} it is known that there exists $\alpha_\lambda\in\big(0,\frac{N-2s}{2}\big]$ such that  any solution $u\in \Ds$ to \eqref{eq:one_pole_pbm_i} satisfies
\begin{equation*}
	|u(x)|\approx C|x|^{\alpha_{\lambda}-\frac{N-2s}{2}}\quad\text{as }|x|\to 0,
\end{equation*}
for some $C>0$ which, in particular, implies that
\begin{equation*}
	\frac{C}{2}|x|^{\alpha_{\lambda}-\frac{N-2s}{2}}\leq u(x)\leq \frac{3}{2}C|x|^{\alpha_{\lambda}-\frac{N-2s}{2}}\quad\text{in a neighborhood of }0.
\end{equation*}
This parameter $\alpha_\lambda$ will play a crucial role in our main result, which we can now state (we refer to \Cref{sec:single} for the details, see in particular \eqref{eq:alpha} for the precise definition).
\begin{theorem}\label{thm:main}
	For $k\geq 2$, let $\bm{\lambda}=(\lambda_1,\dots,\lambda_k)\in\R^k$ and $\bm{a}=(a_1,\dots,a_k)\in\R^{kN}$ be such that $Q_{\LA}$ is positive definite, i.e.
	\begin{equation}\label{eq:positivity}
		\mu_{\LA}>0
	\end{equation}
	where $\mu_{\LA}$ is as in \eqref{eq:mu}, and such that
	\begin{align}
		&\lambda_1\leq \dots\leq \lambda_k<h_{N,s}, \label{eq:main_order}\\
		& \lambda_k>0, \label{eq:main_lambda_k}\\
		&\sum_{i<k}\lambda_i\leq 0,\label{eq:main_sum} \\
		&\left\{ \begin{aligned}
			&\sum_{i< k}\frac{\lambda_i}{|a_k-a_i|^{2\alpha_{\lambda_k}}}>0, &&\text{if }s>\alpha_{\lambda_k}, \\
			&\sum_{i< k}\frac{\lambda_i}{|a_k-a_i|^{2s}}>0, &&\text{if }s\leq \alpha_{\lambda_k},
		\end{aligned}\right.\label{eq:main_cases}
	\end{align}
	where $\alpha_{\lambda_k}$ is defined in \eqref{eq:alpha} (see also the discussion above).
	Then there exists a minimizer for \eqref{eq:min_S_LA} such that
	\begin{enumerate}
		\item $u$ and $\nabla u$ are locally H\"older continuous in $\R^N\setminus\{a_1,\dots,a_k\}$;
		\item $u>0$ in $\R^N\setminus\{a_1,\dots,a_k\}$.
	\end{enumerate}
\end{theorem}

Let us make some comments about assumption \eqref{eq:positivity} which, compared to \eqref{eq:main_order}, \eqref{eq:main_lambda_k}, \eqref{eq:main_sum}, \eqref{eq:main_cases}, is more implicit. A first observation in this direction is that a necessary condition on the masses and on the poles for the existence of solutions of \eqref{eq:problem} is the quadratic form $Q_{\LA}$ to be nonnegative. 
\begin{lemma}\label{lemma:Q_nonnegative}
	Let $\bm{\lambda}=(\lambda_1,\dots,\lambda_k)\in\R^k$ and $\bm{a}=(a_1,\dots,a_k)\in\R^{kN}$ be such that there exists a weak solution $u\in\Ds\setminus\{0\}$ of \eqref{eq:problem}. Then $Q_{\LA}(v)\geq 0$ for all $v\in\Ds$. Hence, if $\mu_{\LA}<0$ for some $\bm{\lambda}=(\lambda_1,\dots,\lambda_k)\in\R^k$ and $\bm{a}=(a_1,\dots,a_k)\in\R^{kN}$ (with $\mu_{\LA}$ as in \eqref{eq:mu}), then \eqref{eq:problem} does not admit solutions.
\end{lemma}
The result in \Cref{lemma:Q_nonnegative} is a standard phenomenon that goes under the name of \emph{Agmon-Allegretto-Piepenbrink principle} and it is a straightforward consequence of the discrete Picone inequality, see e.g. \cite{Frank2008,brasco}. For the sake of completeness we provide a proof of it at the end of \Cref{sec:conclusion}. 

Second, we point out that actually there are explicit conditions that guarantee the positivity/negativity of the quadratic form $Q_{\LA}$. In particular, we recall from \cite{Felli-Roberto} the two following results in this direction. The first one provides a necessary and sufficient condition on the masses $\{\lambda_i\}_{i=1,\dots,k}$ for the existence of a configuration of poles for which \eqref{eq:positivity} holds.
\begin{theorem}[\cite{Felli-Roberto}, Theorem 1.4]\label{thm:FMO1}
	Let $\bm{\lambda}=(\lambda_1, \dots \lambda_k) \in \R^k$. We have the following:
	\begin{itemize}
		\item if $\mu_{\LA}>0$ for some $\bm{a}=(a_1,\dots,a_k)\in\R^{kN}$, then
			\begin{equation*}
			\lambda_i<h_{N,s} \quad\text{for all }i=1,\dots, k, \quad\text{and}\quad \sum_{i=1}^{k}\lambda_i<h_{N,s}.
		\end{equation*}
		\item if there holds
		\begin{equation*}
			\lambda_i<h_{N,s} \quad\text{for all }i=1,\dots, k, \quad\text{and}\quad \sum_{i=1}^{k}\lambda_i<h_{N,s}
		\end{equation*}
		then there exists $\bm{a}=(a_1,\dots,a_k)\in\R^{kN}$ such that $\mu_{\LA}>0$.
	\end{itemize}
\end{theorem}
On the other hand, we also recall from \cite{Felli-Roberto} this other result, which provides:
\begin{itemize}
	\item a sufficient condition on the masses such that \eqref{eq:positivity} holds for any configuration of poles;
	\item a sufficient condition for the negativity of $\mu_{\LA}$ which, in view of \Cref{lemma:Q_nonnegative} implies non-existence of solutions.\footnote{In \cite[Theorem 1.5]{Felli-Roberto} the authors claim that there exists a configuration poles such that $\mu_{\LA}\leq 0$. However, by scanning the proof one easily sees that is proved that $\mu_{\LA}<0$.}
\end{itemize}
  
\begin{theorem}[\cite{Felli-Roberto}, Theorem 1.5]\label{thm:FMO2}
	If $\bm{\lambda}=(\lambda_1,\dots,\lambda_k)\in\R^k$ satisfies
	\begin{equation*}
		\sum_{\substack{i=1 \\ \lambda_i>0}}^k\lambda_i<h_{N,s},
	\end{equation*}
	then the quadratic form $Q_{\LA}$ is positive definite, i.e. $\mu_{\LA}>0$, for all $(a_1,\dots,a_k)\in \R^{kN}$. On the other hand, if $\bm{\lambda}=(\lambda_1,\dots,\lambda_k)\in\R^k$ satisfies
	\[
	\sum_{\substack{i=1 \\ \lambda_i>0}}^k\lambda_i>h_{N,s}
	\]
	then there exists a configuration of poles $\bm{a}=(a_1,\dots,a_k)\in\R^{kN}$ such that $Q_{\LA}$ is not non-negative, i.e. $\mu_{\LA}<0$. In particular, \eqref{eq:problem} does not admit solutions.
\end{theorem}
Then, by following the arguments in \cite[Theorem 1.4]{Felli-Roberto}, one can get a sufficient condition on the masses for the non-existence of solutions for any configuration of poles (see \Cref{sec:conclusion} for the proof).
\begin{theorem}\label{thm:nonex}
	Let $k\geq 2$ and let $\bm{\lambda}=(\lambda_1,\dots,\lambda_k)\in\R^k$. If
	\begin{equation*}
		\lambda_i>h_{N,s}~\text{for some }i=1,\dots,k\quad\text{or}\quad \sum_{i=1}^k\lambda_i>h_{N,s}
	\end{equation*}
	then for any $\bm{a}=(a_1,\dots,a_k)\in\R^{kN}$ there exists $u\in\Ds$ such that $Q_{\LA}(u)<0$. In particular, $\mu_{\LA}<0$ and \eqref{eq:problem} does not admit solutions.
\end{theorem}

Finally, our last main result describes conditions on the masses (which do not fit into the assumptions of \Cref{thm:nonex}) for which we can tell there are no minimizers for \eqref{eq:min_S_LA}.
\begin{theorem}\label{thm:non_achieved}
	Let $k\geq 2$ and let $\bm{\lambda}=(\lambda_1,\dots,\lambda_k)\in\R^k$ be such that
	\begin{equation}\label{eq:non_ex_th1}
		\lambda_1<0\quad\text{and}\quad\lambda_i\leq 0\quad\text{for all }i=1,\dots,k-1
		\quad\text{and}\quad 0< \lambda_k< h_{N,s}
	\end{equation}
	or
	\begin{equation}\label{eq:non_ex_th2}
		0\leq \lambda_i<h_{N,s}\quad\text{for all }i=1,\dots,k\quad\text{and}\quad \sum_{i=1}^k\lambda_i< h_{N,s}.
	\end{equation}
	Then, for any $\bm{a}=(a_1,\dots,a_k)\in\R^{kN}$ the value $S_{\LA}$ is not achieved.
\end{theorem}

To conclude, we provide a simple example of a configuration of masses/poles in which, thanks to the previous results, we know that there exists a minimizer.
\begin{example}[3 poles]
	Let $k=3$. Let us take $\lambda\in(0,h_{N,s}/2)$ and let
	\begin{equation*}
		\lambda_1=-\lambda,\quad\lambda_2=\lambda_3=\lambda.
	\end{equation*}
	Moreover, let 
	\begin{equation*}
		a_1=T\bm{e}_1,\quad a_2=t\bm{e}_1,\quad a_3=0,
	\end{equation*}
	for $T>t>0$. We have that \Cref{thm:FMO2} (first part) ensures that $Q_{\LA}$ is positive definite, i.e. \eqref{eq:positivity}, while one can directly check that \eqref{eq:main_order}--\eqref{eq:main_cases} are satisfied. Therefore, there exists a minimizer for $S_{\LA}$, hence a positive solution to \eqref{eq:problem}.
\end{example}
In general, a minimizer exists every time the following happen:
\begin{itemize}
	\item the $k$-th singularity $\lambda_k$ is positive but lower than the threshold $h_{N,s}$ (so that \eqref{eq:main_lambda_k} holds);
	\item the positive singularities are located near the $k$-th (so that  \eqref{eq:positivity} holds) but they are not too strong (so that the first part of \Cref{thm:FMO2} applies);
	\item the negative singularities are strong enough (such that \eqref{eq:main_sum} holds) but they are located far away from $a_k$ (so that \eqref{eq:main_cases} holds).
\end{itemize}

\medskip

Let us make some comments concerning our main results and let us provide some context. First of all, we point out that the problem \eqref{eq:problem} we are investigating presents a doubly-critical behavior: on one hand, due to the presence of the power $2^*_s$, which is the critical threshold for the Sobolev embedding and which is invariant under rescalings of the type $u_r^{x_0}(x):=r^{-(N-2s)/2}u((x-x_0)/r)$; on the other hand, due to the presence of the multi-polar Hardy-type potential, which behaves as follows
\begin{equation*}
		\int_{\R^N}H_{\LA}(u_r^{a_j})^2\dx=\begin{bvp}
			&\lambda_j\int_{\R^N}\frac{u^2}{|x|^{2s}}\dx+o(1) &&\text{as }r\to 0, \\
			&\Bigg(\sum_{i=1}^k\lambda_i\Bigg)\int_{\R^N}\frac{u^2}{|x|^{2s}}\dx+o(1)&&\text{as }r\to +\infty.
		\end{bvp}
\end{equation*}
In particular, it is \enquote{almost} invariant with respect to such rescalings  and again there is no compactness in the embedding
\begin{equation*}
	\Ds\hookrightarrow L^2(\R^N;|x|^{-2s}\dx).
\end{equation*}
Moreover, it cannot be regarded as a lower-order perturbation with respect to the fractional Laplacian, since it scales the same with respect to dilations, i.e.
\begin{equation*}
	((-\Delta)^s -H_{\LA}(x))(u(rx))=r^{2s}((-\Delta)^su -H_{\bm{\lambda},r\bm{a}}u)(rx).
\end{equation*}
 Both these terms are hence critical and cause lack of compactness of minimizing sequences, which we must then recover by other means, when possible (indeed, we recall that there are explicit cases in which this cannot be done and minimizers do not exist, see \Cref{thm:nonex} and \Cref{thm:non_achieved}). 
 
 \medskip
 
 The operator $\mathcal{L}_{\LA}$ (actually, a translation of it) naturally emerges in relativistic quantum mechanics. Indeed, if one considers, for simplicity the case of one pole, then the Hamiltonian
 \begin{equation*}
 	(-\Delta+m^2)^s-\frac{Ze^2}{|x|^{2s}}
 \end{equation*}
 describes a spin zero relativistic particle of charge $e$ and mass $m$ in the Coulomb
 field of an infinitely heavy nucleus of charge $Z$ (in particular, the relevant case is $s=1/2$), see e.g. \cite{Herbst1977,lieb,lieb_loss}. Potentials with rate of decay $|x|^{-2s}$ are important in relativistic quantum mechanics as they serve as a bridge between regular potentials, which have stationary states, and singular potentials, which have unbounded energy and cause particles to fall to the center. 
 
\medskip

Differential operators with critical Hardy potentials have been widely investigated in the literature. In the case of a single pole, among many others, we quote \cite{Dipierro2016,fall2013,fall2020} (see also references therein) in which the authors study questions of existence or non-existence of solutions to certain PDEs driven by the operator
\begin{equation*}
	(-\Delta)^s-\frac{\lambda}{|x|^{2s}},
\end{equation*}
and \cite{CGHMV,MOV} (see also references therein) in which the authors study uniqueness or non-uniqueness of solutions. However, the case of multiple singularities is much less studied. In the local case $s=1$, the literature treating problems in the same spirit as the present paper is richer. Among many others, we mention \cite{FT2006,FT,Felli2007,FMT_anis,Felli_anis}, which investigated various aspects of equations with multipolar Hardy-type potentials and critical nonlinearities (for instance, the case of anisotropic potentials).

\medskip

We remark that, despite our result is the fractional counterpart of \cite{FT}, the methods we use in order to obtain them required several modifications. Let us here emphasize the main differences.

First, one immediately notices the non-local nature of the fractional Laplacian, in stark contrast with the pure local behavior of the standard Laplacian. In particular, this prevents us from directly using cut-off arguments and blow-up analysis, which are crucial steps in \cite{FT}. In order to overcome this difficulty, we exploit an extension procedure which was first established in \cite{molchanov_ext} and then popularized by the celebrated work by Caffarelli and Silvestre \cite{Caffarelli2007}. Roughly speaking, for any function $u\colon \R^N\to \R$ we consider an extension of it $U\colon \R^{N+1}_+\to \R$ to the upper half-space 
\begin{equation*}
	\R^{N+1}_+:=\{(x,t)\colon x\in\R^N,~t>0\}.
\end{equation*}
The extension can be chosen in such a way that
\begin{equation*}
	\begin{bvp}
		&U(x,0)=u(x),&&\text{for }x\in\R^N, \\
		&\lim_{t\to 0^+}t^{1-2s}\partial_t U(x,t)=\kappa_s (-\Delta)^su(x),&&\text{for }x\in \R^N,
	\end{bvp}
\end{equation*}
so that the fractional Laplacian can be seen as a Dirichlet-to-Neumann operator for a suitable \emph{local} operator acting on $\R^{N+1}_+$, with $\kappa_s>0$ as in \Cref{eq:kappa_s}. We refer to \Cref{sec:ext} for the details.

Another remarkable difference with \cite{FT} is that in \cite{FT} the authors strongly relied on the fact that solutions to the equation with a single pole
\begin{equation}\label{eq:one_pole_loc}
	\begin{bvp}
		-\Delta u-\frac{\lambda}{|x|^2}u&=u^{2^*-1}, &&\text{in }\R^N, \\
		u&>0, &&\text{in }\R^N\setminus\{0\}
	\end{bvp}
\end{equation}
are classified and explicitly known. More precisely, in \cite{terracini_hardy} the author proved that if $u$ solves \eqref{eq:one_pole_loc}, then it is of the form
\[u(x)=\mu^{-\frac{N-2}{2}}w^\lambda\left(\frac{x}{\mu}\right),\quad\text{for some }\mu>0,\]
where
\[w^\lambda(x)=\frac{(N(N-2)\nu_\lambda^2)^{\frac{N-2}{4}}}{(|x|^{1-\nu_\lambda}(1+|x|^{2\nu_\lambda}))^{\frac{N-2}{2}}}\]
and
\[\nu_\lambda:=\sqrt{1-\frac{4\lambda}{N-2}}.\]
On the other hand, in the fractional scenario, the situation is much more obscure. Indeed, the solutions to \eqref{eq:one_pole_pbm_i} are not classified. Furthermore, even if it is conjectured that they should be of the form
\[u(x)=\mu^{-\frac{N-2s}{2}}w^\lambda_s\left(\frac{x}{\mu}\right),\quad\text{for some }\mu>0,\]
\begin{equation}\label{eq:w_lambda}
	w^\lambda_s(x)=\frac{K_{N,s,\lambda}}{(|x|^{1-\alpha_\lambda}(1+|x|^{2\alpha_\lambda}))^{\frac{N-2s}{2}}},
\end{equation}
with $\alpha_\lambda=\alpha_{\lambda,s}>0$ as in \eqref{eq:alpha} and $K_{N,s,\lambda}>0$, we were not able to find any reference proving that $w^\lambda_s$ solves \eqref{eq:one_pole_pbm_i}. Indeed, such a proof presents highly non-trivial technical difficulties, due to the necessity of computing the fractional Laplacian (which is a non-local operator) of functions like \eqref{eq:w_lambda}. In order to overcome such obstacle, we proceed as follows:
\begin{enumerate}
	\item we notice that, in the range $0\leq \lambda<h_{N,s}$, there exists a solution to \eqref{eq:one_pole_pbm_i} (this was proved in \cite[Theorem 1.5]{Dipierro2016});
	\item exploiting the local asymptotics of solutions proved in \cite[Theorem 1.1]{Fall2014} and the upper and lower bounds proved in \cite[Thorem 1.7]{Dipierro2016}, we are able to detect the sharp asymptotic behavior of the solutions of \eqref{eq:one_pole_pbm_i} at the origin and at infinity (see \Cref{lemma:asy}).
\end{enumerate}
These two ingredients turn out to be sufficient for testing the Rayleigh quotient corresponding to \eqref{eq:one_pole_pbm_i} and obtain sharp upper bounds on the energy level (which is then a crucial ingredient in order to apply the concentration compactness argument). In particular such bounds (see \Cref{lemma:PS_c}) involve the following quantity
\begin{equation}\label{eq:S_lambda}
	S_{\lambda}:=\inf_{u\in\Ds\setminus\{0\}}\frac{\displaystyle\norm{u}_{\Ds}^2-\lambda\int_{\R^N}\frac{u^2}{\abs{x}^{2s}}\dx}{\norm{u}_{L^{2^*_s}(\R^N)}^2}
\end{equation}
which, in case $\lambda=0$ coincides with the best constant $S>0$ in the fractional Sobolev inequality, i.e.
\begin{equation}\label{eq:sobolev}
	S_0=S:=\inf_{u\in\Ds\setminus\{0\}}\frac{\displaystyle\norm{u}_{\Ds}^2}{\norm{u}_{L^{2^*_s}(\R^N)}^2},
\end{equation}
see \cite{Cotsiolis2004}. Finally, we denote
\begin{equation}\label{eq:sigma_lambda}
	\sigma_{\bm{\lambda}}:=\sum_{i=1}^k \lambda_i.
\end{equation}
 
\section{The extended formulation}\label{sec:ext}
In this section, we recall the nowadays standard extended formulation of the fractional Laplacian as a Dirichlet-to-Neumann operator. We refer to \cite{Caffarelli2007} for the original paper and to \cite{COR} for detailed proofs. We consider the space $\mathcal{D}^{1,2}(\overline{\R^{N+1}_+};t^{1-2s})$ defined as the completion of $C_c^\infty(\overline{\R^{N+1}_+})$ with respect to the norm induced by the scalar product
\begin{equation*}
	(u,v)_{\mathcal{D}^{1,2}_t(\overline{\R^{N+1}_+})}=\int_{\R^{N+1}_+}t^{1-2s}\nabla u\cdot\nabla v\dxdt.
\end{equation*}
In the following, we use the notation $\mathcal{D}^{1,2}_t(\overline{\R^{N+1}_+})$ in place of $\mathcal{D}^{1,2}(\overline{\R^{N+1}_+};t^{1-2s})$. We have that, in view of the validity of a Sobolev inequality
\begin{equation}\label{eq:sobolev_ext}
	S_{\textup{e}}\left( \int_{\R^{N+1}_+}t^{1-2s}\abs{u}^{2\gamma}\dxdt \right)^{\frac{1}{\gamma}}\leq \int_{\R^{N+1}_+}t^{1-2s}\abs{\nabla u}^2\dxdt\quad\text{for all }u\in \Dext
\end{equation}
where $S_{\textup{e}}>0$ and
\[
\gamma:=1+\frac{2}{N-2s},
\]
the space $\mathcal{D}^{1,2}_t(\overline{\R^{N+1}_+})$ can be characterized as a concrete functional space as follows
\begin{equation}\label{eq:dext_concr}
	\Dext=\left\{u\in L^{2\gamma}(\R^{N+1}_+;t^{1-2s})\colon \norm{u}_{\Dext}<+\infty \right\}.
\end{equation}
For the proof of \eqref{eq:sobolev_ext} we refer to \cite[Proposition 3.3]{Valdinoci-Dipierro} and for \eqref{eq:dext_concr} to \cite[Proposition 2.3]{COR}. Moreover, we have that there exists a linear, continuous and surjective trace operator
\begin{equation*}
	\mathrm{Tr}\colon \Dext\hookrightarrow \Ds
\end{equation*}
such that $(\mathrm{Tr}\, u)(x)=u(x,0)$ for all $u\in C_c^\infty(\overline{\R^{N+1}_+})$, see e.g. \cite[Proposition 2.5 \& 2.6]{COR}. In the following, for any $u\in\Dext$ and $\varphi\in\Ds$, we say that
\begin{equation*}
	u=\varphi\quad\text{on }\R^N
\end{equation*} 
if $\mathrm{Tr}\,u=\varphi$. Now, with these tools at hands, we can formulate the extension procedure. For any fixed $\varphi\in \Ds$, we consider the following minimization problem
\begin{equation*}
	\min\left\{\int_{\R^{N+1}_+}t^{1-2s}|\nabla u|^2\dxdt\colon u\in\Dext,~u=\varphi~\text{on }\R^N\right\}.
\end{equation*}
By classical methods in the calculus of variations, we know that this problem admits a unique solution $U_{\varphi}\in\Dext$ which weakly solves
\begin{equation*}
	\begin{bvp}
		-\dive(t^{1-2s}\nabla U_{\varphi})&=0,&&\text{in }\R^{N+1}_+, \\
		U_{\varphi}&=\varphi, &&\text{on }\R^N, \\
		-\lim_{t\to 0+}t^{1-2s}\partial_t U_{\varphi}&=\kappa_s(-\Delta)^s\varphi, &&\text{on }\R^N,
	\end{bvp}
\end{equation*}
that is
\begin{equation}\label{eq:extension_weak}
	(U_{\varphi},v)_{\Dext}=\kappa_s (\varphi,v)_{\Ds}\quad\text{for all }v\in \Dext,
\end{equation}
where
 \begin{equation}\label{eq:kappa_s}
	\kappa_s:=\frac{\Gamma(1-s)}{2^{2s-1}\Gamma(s)}.
\end{equation}
Moreover, by testing \eqref{eq:extension_weak} with $v=u$ we get that
\begin{equation*}
	\norm{U_{\varphi}}_{\Dext}^2=\kappa_s\norm{\varphi}_{\Ds}^2.
\end{equation*}
Now that the functional setting is fully established, we can provide the extended formulation of the problem we are interested in. More precisely, we have that
\begin{equation*}
	Q_{\LA}(u,v)=\kappa_s^{-1} (U_u,U_v)_{\Dext}-\int_{\R^N}H_{\LA}uv\dx
\end{equation*}
and so we can easily obtain the following characterization
\begin{equation*}
	S_{\LA}=\inf_{\substack{U\in \Dext \\ \Tr U\neq 0}}\frac{\displaystyle \int_{\R^{N+1}_+}t^{1-2s}|\nabla U|^2\dxdt-\kappa_s\sum_{i=1}^k\lambda_i\int_{\R^N}\frac{u^2}{|x-a_i|^{2s}}\dx}{\displaystyle\kappa_s \left(\int_{\R^N}|\Tr U|^{2^*_s}\dx\right)^{\frac{2}{2^*_s}}}.
\end{equation*}
On the other hand, solutions of \eqref{eq:problem} may be regarded as unconstrained critical points of the following functional
\begin{align*}
	G(u):&=\frac{C_{N,s}}{4}\int_{\R^{2N}}\frac{\abs{u(x)-u(y)}^2}{\abs{x-y}^{N+2s}}\dxdy -\frac{1}{2}\int_{\R^N}H_{\LA}u^2\dx-\frac{S_{\LA}}{2^*_s}\int_{\R^N}\abs{u}^{2^*_s}\dx, \\
	&=\frac{1}{2}\norm{u}_{\Ds}^2-\sum_{i=1}^k\lambda_i\norm{\frac{u}{\abs{x-a_i}^s}}_{L^2(\R^N)}^2-\frac{S_{\LA}}{2^*_s}\norm{u}_{L^{2^*_s}(\R^N)}^{2^*_s}.
\end{align*}
defined for $u\in\Ds$. Moreover, if we denote
\begin{equation*}\label{eq:J_ext}
	J(U):=\frac{1}{2}\int_{\R^{N+1}_+}t^{1-2s}|\nabla U|^2\dxdt-\frac{\kappa_s}{2}\int_{\R^N}H\Tr U^2\dx-\kappa_s\frac{\mathcal{S}_{\LA}}{2^*_s} \int_{\R^N}
	|\Tr U|^{2^*_s}\dx,
\end{equation*}
defined for $U\in\Dext$, in view of the extension procedure explained above, we have
\begin{equation*}
	J(U)=\kappa_s G(\Tr u)\quad\text{for all }U\in\Dext.
\end{equation*}

\section{Compactness of Palais-Smale sequences}

In view of the Ekeland Variational Principle \cite{ekeland}, we know that it is not restrictive to consider minimizing sequences of \eqref{eq:min_S_LA} which satisfy the Palais-Smale condition at level $S_{\LA}$. Hence, in the present section, we prove compactness of Palais-Smale sequences under suitable assumption on the energy level. We work with the extended formulation of the problem established in \Cref{sec:ext} and crucial tool is the concentration-compactness principle established in \cite{Valdinoci-Dipierro}.

Hereafter, we denote by $C_0(\overline{\R^{N+1}_+})$ the closure of $C_c^\infty(\overline{\R^{N+1}_+})$ in $L^\infty(\R^{N+1}_+)$. Moreover, we define the space of finite measures on $\overline{\R^{N+1}_+}$, hereafter denoted as $\mathcal{M}(\overline{\R^{N+1}_+})$, as the space of linear continuous functionals on $C_0(\overline{\R^{N+1}_+})$ and, with a little abuse of notation, for $\mu\in \mathcal{M}(\overline{\R^{N+1}_+})$ and $\phi\in C_0(\overline{\R^{N+1}_+})$ we may write
\begin{equation*}
	\int_{\overline{\R^{N+1}_+}}\phi\d \mu\qquad\text{in place of}\qquad	_{C_0(\overline{\R^{N+1}_+})}\langle \phi,\mu\rangle_{\mathcal{M}(\overline{\R^{N+1}_+})}.
\end{equation*}
We also consider the subspace of nonnegative finite measures
\begin{multline*}
	\mathcal{M}^+(\overline{\R^{N+1}_+}):=\Bigg\{\mu\in\mathcal{M}(\overline{\R^{N+1}_+})\colon \int_{\overline{\R^{N+1}_+}} \phi \,\mathrm{d}\mu\geq 0 \\
	\text{ for all }\phi\in C_0(\overline{\R^{N+1}_+})~\text{such that }\phi\geq 0\text{ in }\overline{\R^{N+1}_+}  \Bigg\}.
\end{multline*}
We now recall the notion of weak convergence.
\begin{definition}
	Let $\{\mu_n\}_n\sub\mathcal{M}(\overline{\R^{N+1}_+})$ and $\mu\in\mathcal{M}(\overline{\R^{N+1}_+})$. We say that $\mu_n$ \emph{weakly converges to} $\mu$ in $\mathcal{M}(\overline{\R^{N+1}_+})$ as $n\to \infty$, and we write
	\[
	\mu_n\rightharpoonup\mu \quad\text{weakly in }\mathcal{M}(\overline{\R^{N+1}_+}),\text{ as }n\to\infty,
	\]
	if, for all $\phi\in C_0(\overline{\R^{N+1}_+})$,
	\[
	\int_{\overline{\R^{N+1}_+}}\phi\d \mu_n\to\int_{\overline{\R^{N+1}_+}}\phi\d \mu,\quad\text{as }n\to\infty.
	\]
\end{definition}
In order to make use of the concentration-compactness principle, we need the notion of tight sequence recalled below.
\begin{definition}
	A sequence $\{u_n\}_n\sub\Dext$ is said to be \emph{tight} if for any $\epsilon>0$ there exists $R>0$ and $n_0\in\N$ such that
	\[
	\int_{\R^{N+1}_+\setminus B_R^+}t^{1-2s}\abs{\nabla u_n}^2\dxdt<\epsilon\quad\text{for all }n\geq n_0.
	\]
\end{definition}

Hereafter, for any $r>0$ we let $\eta_r\colon \overline{\R^{N+1}_+}\to \R$ be a cut-off function such that
\begin{equation}\label{eq:eta}
	\begin{gathered}
		\eta_r\in C^\infty(\overline{\R^{N+1}_+}),\qquad \eta_r=\begin{cases}
			1, &\text{in }\overline{B_r^+}, \\
			0, &\text{in }\overline{\R^{N+1}_+}\setminus B_{2r}^+, \\
		\end{cases} \\
		0\leq \eta_r\leq 1,\qquad \abs{\nabla\eta_r}\leq \frac{2}{r}.
	\end{gathered}
\end{equation}
Moreover, for $x_0\in\R^N=\partial\R^{N+1}_+$, we denote
\begin{equation}\label{eq:eta_t}
	\eta_r^{x_0}(z):=\eta_r(z-x_0).
\end{equation}

Since the proofs in the following involve testing minimization problems and the Palais-Smale condition for a Palais-Smale sequence $\{u_n\}_n\sub\Dext$ with $u_n$ suitably cut-off by multiplication with some $\eta_\epsilon^{x_0}$, we need the following preliminary lemma.

\begin{lemma}
	Let $\{u_n\}_n\sub\Dext $ be bounded and let $\eta_\epsilon^{x_0}$ be as in \eqref{eq:eta_t}. Then, for any $x_0\in\R^N$ there holds
	\begin{equation}\label{eq:claim_0_1}
		\lim_{\epsilon\to 0}\lim_{n\to\infty}\int_{\R^{N+1}_+}t^{1-2s}u_n^2|\nabla \eta_\epsilon^{x_0}|^2\dxdt=0,
	\end{equation}
	from which we then immediately deduce the following
	\begin{align}
		&\lim_{\epsilon\to 0}\lim_{n\to\infty}\int_{\R^{N+1}_+}t^{1-2s}u_n\nabla u_n\cdot\nabla \eta_\epsilon^{x_0}\dxdt=0 \label{eq:claim_0_2}\\
		&\lim_{\epsilon\to 0}\lim_{n\to\infty}\int_{\R^{N+1}_+}t^{1-2s}u_n\eta_\epsilon^{x_0}\nabla u_n\cdot\nabla \eta_\epsilon^{x_0}\dxdt=0.\label{eq:claim_0_3}
	\end{align}
\end{lemma}
\begin{proof}
	We observe that, up to a subsequence (still labeled as $\{u_n\}_n$), we have 
	\begin{align*}
		& u_n\weak u &&\text{weakly in }\Dext,~\text{in }L^{2\gamma}_{\textup{loc}}(\R^{N+1}_+,t^{1-2s})~\text{and in }L^{2^*_s}(\R^N) \\
		& u_n\to u &&\text{strongly in }L^p_{\textup{loc}}(\R^{N+1}_+,t^{1-2s}) ~\text{and in }L^q_{\textup{loc}}(\R^N)
	\end{align*}
	as $n\to\infty$, for all $p\in [1,2\gamma)$ and all $q\in [1,2^*_s)$, for some $u\in\Dext$.	Let us first prove \eqref{eq:claim_0_1}. First of all, we notice that, by the $L^p_{\textup{loc}}$ convergence, we can pass to the limit as $n\to\infty$, i.e.
\begin{align*}
	\lim_{n\to\infty}\int_{\R^{N+1}_+}t^{1-2s}u_n^2|\nabla \eta_\epsilon^{x_0}|^2\dxdt&=\lim_{n\to\infty}\int_{B_{2\epsilon}^+(x_0)}t^{1-2s}u_n^2|\nabla \eta_\epsilon^{x_0}|^2\dxdt \\
	&=\int_{B_{2\epsilon}^+(x_0)}t^{1-2s}u^2|\nabla \eta_\epsilon^{x_0}|^2\dxdt
\end{align*}
At this point, from H\"older inequality and the fact that $|\nabla\eta_\epsilon^{x_0}|\leq 2/\epsilon$ we derive that
\begin{align*}
	\int_{B_{2\epsilon}^+(x_0)}t^{1-2s}u^2|\nabla \eta_\epsilon^{x_0}|^2\dxdt &\leq \left(\int_{B_{2\epsilon}^+(x_0)}t^{1-2s}|u|^{2\gamma}\dxdt\right)^{\frac{1}{\gamma}}\left(\int_{B_{2\epsilon}^+(x_0)}t^{1-2s}|\nabla \eta_\epsilon^{x_0}|^{2\gamma'}\dxdt\right)^{\frac{1}{\gamma'}} \\
	&\leq \left(\int_{B_{2\epsilon}^+(x_0)}t^{1-2s}|u|^{2\gamma}\dxdt\right)^{\frac{1}{\gamma}}\left(2^{2\gamma'}\epsilon^{-2\gamma'}\int_{B_{2\epsilon}^+(x_0)}t^{1-2s}\dxdt\right)^{\frac{1}{\gamma'}} \\
	&\leq C\left(\int_{B_{2\epsilon}^+(x_0)}t^{1-2s}|u|^{2\gamma}\dxdt\right)^{\frac{1}{\gamma}}(\epsilon^{N+2-2s-2\gamma'})^{\frac{1}{\gamma'}} \\
	&\leq C\left(\int_{B_{2\epsilon}^+(x_0)}t^{1-2s}|u|^{2\gamma}\dxdt\right)^{\frac{1}{\gamma}},
\end{align*}
where $C>0$ is independent of $\epsilon$. Hence, by passing to the limit as $\epsilon\to 0$ we get \eqref{eq:claim_0_1}. In order to prove \eqref{eq:claim_0_2} and \eqref{eq:claim_0_3}, we observe that
\begin{align*}
	&\left|\int_{\R^{N+1}_+}t^{1-2s}u_n\eta_\epsilon^{x_0}\nabla u_n\cdot\nabla \eta_\epsilon^{x_0}\dxdt\right|+\left|\int_{\R^{N+1}_+}t^{1-2s}u_n\nabla u_n\cdot\nabla \eta_\epsilon^{x_0}\dxdt\right| \\
	&\leq 2\int_{B_{2\epsilon}^+(x_0)}t^{1-2s}|u_n||\nabla u_n||\nabla \eta_\epsilon^{x_0}|\dxdt \\
	&\leq 2\left(\int_{B_{2\epsilon}^+(x_0)}t^{1-2s}|\nabla u_n|^2\dxdt\right)^{\frac{1}{2}}\left(\int_{B_{2\epsilon}^+(x_0)}t^{1-2s}|u_n|^2|\nabla \eta_\epsilon^{x_0}|^2\dxdt\right)^{\frac{1}{2}}.
\end{align*}
Hence, since the first term is uniformly bounded in $n$ and $\epsilon$, by \eqref{eq:claim_0_1} we conclude the proof of both \eqref{eq:claim_0_2} and \eqref{eq:claim_0_3}.
\end{proof}

We now recall the concentration-compactness principle that we are going to use. The proof of the main part is contained in \cite{Valdinoci-Dipierro} and here we prove just some consequences which we need in the following.
\begin{theorem}\label{thm:conc_comp}
	Let $\{u_n\}_n\sub\Dext$ be a tight sequence such that
	\[
	u_n\rightharpoonup u\quad\quad\text{weakly in }\Dext~\text{as }n\to\infty,
	\]
	for some $u\in\Dext$. Let $\mu$ and $\nu$ be two positive finite Radon measures on $\overline{\R^{N+1}_+}$ and $\R^N$, respectively, and $\{\gamma_i\}_{i=1,\dots,k}$ a family of positive finite Radon measures on $\R^N$ such that
	\begin{align*}
		t^{1-2s}\abs{\nabla u_n}^2&\weak  \mu \quad\text{weakly in }\mathcal{M}(\overline{\R^{N+1}_+}), \\
		\abs{ u_n}^{2^*_s}&\weak \nu \quad\text{weakly in }\mathcal{M}(\R^N), \\
		\frac{ u_n^2}{\abs{x-a_i}^{2s}}&\weak \gamma_i\quad\text{weakly in }\mathcal{M}(\R^N),~\text{for all }i=1,\dots,k,
	\end{align*}
	as $n\to\infty$. Then there exists an at most countable set $\mathcal{J}\sub\N$, a family of points $\{x_j\}_{j\in \mathcal{J}}\sub\R^N\setminus\{a_1,\dots,a_k\}$ and families of coefficients $\{\mu_{a_i}\}_{i=1,\dots,k},~\{\mu_{x_j}\}_{j\in \mathcal{J}},~\{\nu_{a_i}\}_{i=1,\dots,k},~\{\nu_{x_j}\}_{j\in \mathcal{J}},$ $\{\gamma_{a_i}\}_{i=1,\dots,k}\sub\R$ such that $\mu_{x_j},\nu_{x_j}\geq 0$ for all $j\in \mathcal{J}$, $\mu_{a_i},\nu_{a_i},\gamma_{a_i}\geq 0$ for all $i=1,\dots,k$ and
	\begin{enumerate}
		\item[(i)] $\displaystyle \mu\geq t^{1-2s}\abs{\nabla u}^2+\sum_{i=1}^k\mu_{a_i}\delta_{a_i}+\sum_{j\in \mathcal{J}}\mu_{x_j} \delta_{x_j}$,
		\item[(ii)] $\displaystyle \nu=\abs{ u}^{2^*_s}+\sum_{i=1}^k\nu_{a_i}\delta_{a_i}+\sum_{j\in \mathcal{J}}\nu_{x_j}\delta_{x_j}$,
		\item[(iii)] $\displaystyle \gamma_i=\frac{{u}^2}{\abs{x-a_i}^{2s}}+\gamma_{a_i} \delta_{a_i}$, for all $i=1,\dots,k$.
	\end{enumerate}
	In addition, we have that
	\begin{equation}\label{eq:conc_comp_sobolev}
		\kappa_s S \nu_{a_i}^{\frac{2}{2^*_s}}\leq \mu_{a_i}\quad\text{and}\quad \kappa_s S \nu_{x_j}^{\frac{2}{2^*_s}}\leq \mu_{x_j},
	\end{equation}
	for all $i=1,\dots,k$ and all $j\in\mathcal{J}$, where $S>0$ is as in \eqref{eq:sobolev}.
\end{theorem}
\begin{proof}
	The proof of (i) and (ii) is given in \cite[Proposition 3.2.1]{Valdinoci-Dipierro}. In order to prove (iii), we first observe that from weak convergence in $\mathcal{D}_t^{1,2}(\overline{\R^{N+1}_+})$ we deduce strong convergence in $L^2_{\textup{loc}}(\R^N)$. Hence, if $x_0\in\R^N\setminus\{a_1,\dots,a_j\}$ and $r>0$ is such that $B_r(x_0)\sub\R^N\setminus\{a_1,\dots,a_k\}$, then 
	\begin{equation*}
		\frac{u_n^2}{|x-a_i|^{2s}}\to \frac{u^2}{|x-a_i|^{2s}}\quad\text{strongly in }L^2(B_r(x_0))~\text{as }n\to\infty,
	\end{equation*}
	for all $i=1,\dots,k$, from which we deduce that
	\begin{equation*}
		\gamma_i=\frac{u^2}{|x-a_i|^{2s}}\quad\text{in }\R^N\setminus\{a_1,\dots,a_k\},
	\end{equation*}
	thus implying that $\gamma_i$ can concentrate only at $a_i$, thus completing the proof. It remains to prove the first part of \eqref{eq:conc_comp_sobolev}.
	Let us fix $i \in \{1, \dots, k\}$. For any $\epsilon>0$, let $\eta_\epsilon^{a_i}$ be the cut-off function centered at $a_i$ as defined in \eqref{eq:eta_t}.
	By the fractional Sobolev inequality \eqref{eq:sobolev} combined with \eqref{eq:extension_weak}, we have
	\begin{equation}\label{eq:proof_conc_1}
		\kappa_s S \left( \int_{\R^N} |u_n \eta_\epsilon^{a_i}|^{2^*_s} \dx \right)^{\frac{2}{2^*_s}} \leq \int_{\R^{N+1}_+} t^{1-2s} |\nabla (u_n \eta_\epsilon^{a_i})|^2 \dxdt.
	\end{equation}
	Expanding the right-hand side, we obtain
	\begin{multline}\label{eq:proof_conc_2}
		\int_{\R^{N+1}_+} t^{1-2s} |\nabla (u_n \eta_\epsilon^{a_i})|^2 \dxdt = \int_{\R^{N+1}_+} t^{1-2s} |\nabla u_n|^2 |\eta_\epsilon^{a_i}|^2 \dxdt \\
		+ \int_{\R^{N+1}_+} t^{1-2s} u_n^2 |\nabla \eta_\epsilon^{a_i}|^2 \dxdt + 2 \int_{\R^{N+1}_+} t^{1-2s} u_n \eta_\epsilon^{a_i} \nabla u_n \cdot \nabla \eta_\epsilon^{a_i} \dxdt.
	\end{multline}
	We now pass to the limit by first taking $\limsup_{n \to \infty}$ and then $\lim_{\epsilon \to 0}$. 
	For the left-hand side of \eqref{eq:proof_conc_1}, the weak convergence of measures implies
	\begin{equation*}
		\lim_{\epsilon \to 0} \limsup_{n \to \infty} \kappa_s S \left( \int_{\R^N} |u_n \eta_\epsilon^{a_i}|^{2^*_s} \dx \right)^{\frac{2}{2^*_s}} = \kappa_s S \nu_{a_i}^{\frac{2}{2^*_s}}.
	\end{equation*}
	For the right-hand side expanded in \eqref{eq:proof_conc_2}, the second and third integrals vanish exactly in view of \eqref{eq:claim_0_1} and \eqref{eq:claim_0_3}. Consequently, the only surviving term is the first one, which gives
	\begin{equation*}
		\lim_{\epsilon \to 0} \limsup_{n \to \infty} \int_{\R^{N+1}_+} t^{1-2s} |\nabla u_n|^2 |\eta_\epsilon^{a_i}|^2 \dxdt = \lim_{\epsilon \to 0} \int_{\overline{\R^{N+1}_+}} |\eta_\epsilon^{a_i}|^2 \d\mu = \mu_{a_i}.
	\end{equation*}
	Combining these limits into \eqref{eq:proof_conc_1}, we definitively conclude that $\kappa_s S \nu_{a_i}^{\frac{2}{2^*_s}} \leq \mu_{a_i}$.
\end{proof}
We now here recall the definition of Palais-Smale sequence.
\begin{definition}
	We say that $\{u_n\}_n\sub \Dext$ is a \emph{Palais-Smale sequence for $J$ at level $c\in\R$} if
	\begin{enumerate}
		\item[1.] $J(u_n)\to c$ as $n\to\infty$,
		\item[2.] $J'(u_n)\to 0$ in $(\Dext)^*$ as $n\to\infty$.
	\end{enumerate}
\end{definition}

Before starting with the results of the present section, we state three remarks that are useful in the following.
\begin{remark}\label{rmk:tightness}
	One can easily observe that a sequence $\{u_n\}_n\sub\Dext$ is tight if and only if
	\begin{equation}\label{eq:tight_th}
		\lim_{R\to\infty}\limsup_{n \to \infty}\int_{\R^{N+1}_+\setminus B_R^+}t^{1-2s}
		\abs{\nabla u_n}^2\dxdt=0.
	\end{equation}
\end{remark}

\begin{remark}\label{rmk:cutoff}
	We recall the fact that, by means of basic computations and \eqref{eq:sobolev_ext} it follows that, if $u\in\Dext$ and $\phi\in C_c^\infty(\overline{\R^{N+1}_+})$, then $u\phi\in\Dext$ and 
\begin{multline*}
	\norm{u\phi}_{\Dext}^2\leq \bigg( \norm{\phi}_{L^\infty(\R^{N+1}_+)}^2+S_{\textup{e}}^{-1}\norm{\nabla \phi}_{L^\infty(\R^{N+1}_+)}^2 \\
	+\left(\int_{\supp \phi}t^{1-2s}\dxdt\right)^{\frac{2}{N+2-2s}} \bigg)\norm{u}_{\Dext}^2.
\end{multline*}
	In particular, if $\{u_n\}_n\sub\Dext$ is bounded, then so is $\{u_n\phi\}_n$, for any $\phi\in C_c^\infty(\overline{\R^{N+1}_+})$.
\end{remark}

\begin{remark}
	We explicit here the expression of the differential of $J$:
	\begin{equation*}
		\<J'(u),v\>=\int_{\R^{N+1}_+}t^{1-2s}\nabla u\cdot\nabla v\dxdt-\kappa_s\int_{\R^N}H u v\dx \\
		-\kappa_s S_{\LA}\int_{\R^N}\abs{ u}^{2^*_s-2} u v\dx
	\end{equation*}
	for $u,v\in\Dext$ and in particular
	\begin{equation*}\label{eq:J'u}
		\<J'(u),u\>=\int_{\R^{N+1}_+}t^{1-2s}\abs{\nabla u}^2\dxdt-\kappa_s\int_{\R^N}H u^2\dx-\kappa_s S_{\LA}\int_{\R^N}\abs{ u}^{2^*_s}\dx,
	\end{equation*}
	which implies the following relation
	\begin{equation}\label{eq:J'u_u}
		J(u)-\frac{1}{2}\<J'(u),u\>=\frac{s\kappa_s}{N}S_{\LA}\int_{\R^N}\abs{ u}^{2^*_s}.
	\end{equation}
\end{remark}

We now state the first result of this section, which establishes boundedness of Palais-Smale sequences.

\begin{lemma}\label{lemma:bounded}
	Let $\{u_n\}_n\sub\Dext$ be a Palais-Smale sequence for $J$ at level $c\in \R$.	Then $\{u_n\}_n$ is bounded in $\Dext$. In particular
	\begin{equation}\label{eq:bounded_th1}
		\<J'(u_n),u_n\>\to 0\quad\text{as }n\to\infty.
	\end{equation}
\end{lemma}
\begin{proof}
	By definition of $\mu_{\LA}$ (see \eqref{eq:mu}) and simple computations one can see that
\begin{equation}\label{eq:bounded_1}
		\begin{aligned}
		\mu_{\LA}\norm{u_n}_{\Dext}^2 &\leq \norm{u_n}_{\Dext}^2-\kappa_s\int_{\R^N} Hu_n^2\dx \\
		&=\<J'(u_n),u_n\>+\kappa_s S_{\LA}\int_{\R^N}\abs{u_n}^{2^*_s}\dx \\
		&=\frac{N}{s}J(u_n)-\frac{N-2s}{2s}\<J'(u_n),u_n\>.
	\end{aligned}
\end{equation}
	On the other hand, by Cauchy-Schwarz's inequality we obtain that
	\[
		\abs{\<J'(u_n),u_n\>}\leq \norm{J'(u_n)}_{(\Dext)^*}\norm{u_n}_{\Dext}=\norm{u_n}_{\Dext}o(1),\quad\text{as }n\to\infty.
	\]
	This fact, in combination with \eqref{eq:bounded_1}, yields
	\[
		(\mu_{\LA}+o(1))\norm{u_n}_{\Dext}\leq \frac{Nc}{s}+o(1),\quad\text{as }n\to\infty,
	\]
	which, in view of \eqref{eq:positivity}, concludes the proof.
\end{proof}

Second, we provide sufficient conditions ensuring the tightness of Palais-Smale sequences.

\begin{lemma}\label{lemma:tight}
	Let $S_{\LA}$ be as in \eqref{eq:min_S_LA}, $\sigma_{\bm{\lambda}}$ as in \eqref{eq:sigma_lambda} and $S_{\sigma_{\bm{\lambda}}}$ as in \eqref{eq:S_lambda}. Let $\{u_n\}_n\sub\Dext$ be a Palais-Smale sequence for $J$ at level $c\in \R$.	We have that, if
	\begin{equation}\label{eq:tight_hp1}
		c< \frac{s\kappa_s}{N}S_{\LA}^{1-\frac{N}{2s}} \left(S_{\sigma_{\bm{\lambda}}}\right)^{\frac{N}{2s}},
	\end{equation}
	then $\{u_n\}_n$ is tight.
\end{lemma}
\begin{proof}
	In view of \Cref{rmk:tightness}, we prove the tightness of $\{u_n\}_n$ by proving that $\mu_\infty=0$, where
	\[
		\mu_\infty:=\lim_{R\to\infty}\limsup_{n \to \infty}\int_{\R^{N+1}_+\setminus B_R^+}t^{1-2s}\abs{\nabla u_n}^2\dxdt=0.
	\] 
	We also denote
	\begin{equation*}
		\nu_\infty:=\lim_{R\to\infty}\limsup_{n \to \infty}\int_{\R^N\setminus B_R'}\abs{u_n}^{2^*_s}\dx, \qquad 
		\gamma_\infty:=\lim_{R\to\infty}\limsup_{n \to \infty} \int_{\R^N\setminus B_R'}\frac{u_n^2}{\abs{x}^{2s}}\dx.
	\end{equation*}
	First of all, from \eqref{eq:bounded_th1} and \eqref{eq:J'u_u} we know that, as $n\to +\infty$
	\begin{align*}
		c&=J(u_n)+o(1)=J(u_n)-\frac{1}{2}\<J'(u_n),u_n\>+o(1)  \\
		&=\frac{s\kappa_s}{N}S_{\LA}\int_{\R^N}\abs{u_n}^{2^*_s}\dx+o(1) \geq \frac{s\kappa_s}{N}S_{\LA}\int_{\R^N\setminus B_R'}\abs{u_n}^{2^*_s}\dx+o(1),
	\end{align*}
	for all $R>0$. By taking the limit superior as $n\to \infty$ and the limit as $R\to \infty$, we obtain that
	\begin{equation*}
		c\geq \frac{s\kappa_s}{N}S_{\LA}\nu_\infty
	\end{equation*}
	which, combined with \eqref{eq:tight_hp1}, yields
	\begin{equation}\label{eq:tight_1}
		\nu_\infty<\left(\frac{S_{\sigma_{\bm{\lambda}}}}{S_{\LA}}\right)^{\frac{N}{2s}}.
	\end{equation}
	We now claim that
	\begin{equation}\label{eq:tight_2}
		S_{\LA}\nu_\infty\geq S_{\sigma_{\bm{\lambda}}}\nu_\infty^{2/2^*_s}.
	\end{equation}
	 In order to obtain this, let us consider, for any $R>0$, the cut-off function $\xi_R=1-\eta_R$, where $\eta_R$ is as in \eqref{eq:eta}. As a consequence of trivial arguments, one can see that
	\begin{equation}\label{eq:tight_7}
		\begin{gathered}
			\lim_{R\to\infty}\limsup_{n \to \infty} \int_{\R^{N+1}_+}t^{1-2s}\xi_R^\alpha\abs{\nabla u_n}^2\dxdt=\mu_\infty, \\
					\nu_\infty=\lim_{R\to\infty}\limsup_{n \to \infty} \int_{\R^N}\xi_R^\alpha\abs{ u_n}^{2^*_s}\dx\quad\text{and}\quad \gamma_\infty=\lim_{R\to\infty}\limsup_{n \to \infty}\int_{\R^N}\frac{\xi_R^\alpha u_n^2}{\abs{x}^{2s}}\dx,
		\end{gathered}
	\end{equation}
	for all $\alpha\geq 1$. We now prove that
	\begin{equation}\label{eq:tight_3}
		\mu_\infty=\lim_{R\to\infty}\limsup_{n \to \infty}\int_{\R^{N+1}_+}t^{1-2s}\abs{\nabla (\xi_R u_n)}^2\dxdt.
	\end{equation}
	Indeed, we have that
	\begin{multline}\label{eq:tight_6}
		\int_{\R^{N+1}_+}t^{1-2s}\abs{\nabla (\xi_R u_n)}^2\dxdt= \int_{\R^{N+1}_+}t^{1-2s}\xi_R^2\abs{\nabla u_n}^2\dxdt \\+2\int_{\R^{N+1}_+}t^{1-2s}\xi_R u_n \nabla \xi_R\cdot \nabla u_n\dxdt+\int_{\R^{N+1}_+}t^{1-2s}u_n^2\abs{\nabla \xi_R}^2\dxdt,
	\end{multline}
	Since $\{u_n\}_n$ is bounded in $\Dext$ (see Lemma \ref{lemma:bounded}), there exists $u\in\Dext$ such that $u_n\to u$ in $L^p_{\textup{loc}}(\R^{N+1}_+)$ as $n\to\infty$, for any $p\in[1,2^*_s)$. Therefore
	\begin{equation}\label{eq:tight_4}
			\begin{aligned}
			\limsup_{n \to \infty} \int_{\R^{N+1}_+}t^{1-2s}u_n^2\abs{\nabla \xi_R}^2\dxdt&=\limsup_{n \to \infty}\int_{B_{2R}^+\setminus B_R^+}t^{1-2s}u_n^2\abs{\nabla \xi_R}^2\dxdt \\
			&=\int_{B_{2R}^+\setminus B_R^+}t^{1-2s}u^2\abs{\nabla \xi_R}^2\dxdt,
		\end{aligned}
	\end{equation}
	where we also used the fact that $\nabla \xi_R\in L^\infty(\R^{N+1}_+)$. Now, by H\"older's inequality and \eqref{eq:eta} we have that
	\begin{align*}
		&\int_{B_{2R}^+\setminus B_R^+}t^{1-2s} u_n^2\abs{\nabla \xi_R}^2\dxdt \\
		& \leq \left(\int_{B_{2R}^+\setminus B_R^+}t^{1-2s}\abs{u_n}^{2\gamma} \dxdt \right)^{\frac{1}{\gamma}}\left(\int_{B_{2R}^+\setminus B_R^+} t^{1-2s}\abs{\nabla \xi_R}^{N+2-2s}\dxdt \right)^{\frac{2}{N+2-2s}} \\
		&\leq  \left(\int_{B_{2R}^+\setminus B_R^+}t^{1-2s}\abs{u_n}^{2\gamma} \dxdt \right)^{\frac{1}{\gamma}}\left(\int_{B_2^+}t^{1-2s}\dxdt\right)^{\frac{2}{N+2-2s}}.
	\end{align*}
	This, together with \eqref{eq:tight_4} and Sobolev's inequality \eqref{eq:sobolev_ext}, implies that
	\begin{equation}\label{eq:tight_5}
		\lim_{R\to\infty}\limsup_{n \to \infty}\int_{\R^{N+1}_+}t^{1-2s}u_n^2\abs{\nabla \xi_R}^2\dxdt=0.
	\end{equation}
	Now, by Cauchy-Schwarz's inequality, one can observe that
	\begin{multline*}
		\abs{\int_{\R^{N+1}_+}t^{1-2s}\xi_R u_n\nabla \xi_R\cdot\nabla u_n\dxdt} \\
		\leq \left( \int_{\R^{N+1}_+}t^{1-2s}\xi_R^2\abs{\nabla u_n}^2\dxdt \right)^{\frac{1}{2}}\left(\int_{\R^{N+1}_+} t^{1-2s}u_n^2\abs{\nabla \xi_R}^2\dxdt \right)^{\frac{1}{2}}.
	\end{multline*}
	Hence, thanks to boundedness of $\{u_n\}_n$ and \eqref{eq:tight_5} there holds
	\begin{equation}\label{eq:tight_9}
		\lim_{R\to\infty}\limsup_{n \to \infty}\int_{\R^{N+1}_+}t^{1-2s}\xi_R u_n\nabla \xi_R\cdot\nabla u_n\dxdt=0.
	\end{equation}
	Consequently, by \eqref{eq:tight_6} we have that
	\begin{equation*}
		\lim_{R\to\infty}\limsup_{n \to \infty}\int_{\R^{N+1}_+}t^{1-2s}\abs{\nabla (\xi_R u_n)}^2\dxdt=\lim_{R\to\infty}\limsup_{n \to \infty} \int_{\R^{N+1}_+}t^{1-2s}\xi_R^2\abs{\nabla u_n}^2\dxdt.
	\end{equation*}
	Then, by taking into account \eqref{eq:tight_7}, \eqref{eq:tight_3} is proved. Now, since $\xi_R u_n\in\Dext$ for any $n\in\N$ and $R>0$, by definition of $S_{\sigma_{\bm{\lambda}}}$ we have
	\begin{equation*}
		\kappa_s S_{\sigma_{\bm{\lambda}}}\left(\int_{\R^N}\abs{\xi_R u_n}^{2^*_s}\dx \right)^{\frac{2}{2^*}}\leq \int_{\R^{N+1}_+}t^{1-2s}\abs{\nabla (\xi_R u_n)}^2\dxdt-\kappa_s\left(\sum_{i=1}^k\lambda_i\right)\int_{\R^N}\frac{(\xi_R u_n)^2}{\abs{x}^{2s}}\dx.
	\end{equation*}
	From this, \eqref{eq:tight_7} and \eqref{eq:tight_3} one can easily see that
	\begin{equation}\label{eq:tight_10}
		\kappa_s S_{\sigma_{\bm{\lambda}}}\nu_\infty^{\frac{2}{2^*_s}}\leq \mu_\infty-\kappa_s\left(\sum_{i=1}^k\lambda_i\right)\gamma_\infty.
	\end{equation}
	We now prove that
	\begin{multline}\label{eq:tight_8}
		\lim_{R\to\infty}\limsup_{n \to \infty}\sum_{i=1}^k\lambda_i\int_{\R^N}\frac{\xi_R u_n^2}{\abs{x-a_i}^{2s}}\dx\\
		=\lim_{R\to\infty}\limsup_{n \to \infty}\left(\sum_{i=1}^k \lambda_i\right)\int_{\R^N}\frac{\xi_R u_n^2}{\abs{x}^{2s}}\dx=\left(\sum_{i=1}^k \lambda_i\right)\gamma_\infty.
	\end{multline}
	Basically, this is a consequence of the fact that
	\[
		\frac{1}{\abs{x-a_i}^{2s}}\sim \frac{1}{\abs{x}^{2s}}\quad\text{as }\abs{x}\to \infty.
	\]
	More precisely, by explicit computations we derive that there exists a constant $C>0$ such that
	\[
		\abs{\frac{1}{\abs{x-a_i}^{2s}} -\frac{1}{\abs{x}^{2s}} }\leq \frac{C}{\abs{x}^{2s+1}}\quad\text{for all }\abs{x}\geq 1~\text{and all }i=1,\dots,k.
	\]
	Furthermore, by H\"older's inequality and boundedness of $\{u_n\}_n$ in $L^{2^*_s}(\R^N)$
	\[
		\int_{\R^N}\frac{\xi_R u_n^2}{\abs{x}^{2s+1}}\dx\leq \left(\int_{\R^N\setminus B_{2R}'}\abs{u_n}^{2^*_s}\dx\right)^{\frac{2}{2^*_s}}\left( \int_{\R^N\setminus B_{2R}'}\abs{x}^{-\frac{(2s+1)N}{2s}}\dx \right)^{\frac{2s}{N}}\to 0
	\]
	as $R\to \infty$, uniformly in $n$ and this implies \eqref{eq:tight_8}.	On the other hand, being $\{\xi_R u_n\}_n$ bounded in $\Dext$ (this easily follows from boundedness of $\{u_n\}_n$ and basic estimates), we have that
	\begin{align*}
		0&=\lim_{R\to\infty}\limsup_{n \to \infty} \<J'(u_n),\xi_R u_n\> \\
		&=\lim_{R\to\infty}\limsup_{n \to \infty} \bigg[ \int_{\R^{N+1}_+}t^{1-2s}u_n\nabla u_n\cdot\nabla \xi_R\dxdt+\int_{\R^{N+1}_+}t^{1-2s}\xi_R\abs{\nabla u_n}^2\dxdt, \\
		&\hspace{3cm} -\kappa_s\sum_{i=1}^k\lambda_i\int_{\R^N}\frac{\xi_R u_n^2}{\abs{x-a_i}^{2s}}\dx-\kappa_s S_{\LA}\int_{\R^N}\xi_R\abs{u_n}^{2^*_s}\dx\bigg].
	\end{align*}
	Hence, in view of \eqref{eq:tight_7}, \eqref{eq:tight_8} and the fact that
	\[
		\lim_{R\to\infty}\limsup_{n \to \infty} \int_{\R^{N+1}_+}t^{1-2s}u_n\nabla u_n\cdot\nabla \xi_R\dxdt=0,
	\]
	which is a trivial consequence of \eqref{eq:tight_9}, we obtain that
	\begin{equation}\label{eq:tight_11}
		\mu_\infty-\kappa_s \left(\sum_{i=1}^k\lambda_i\right)\gamma_\infty-\kappa_s S_{\LA}\nu_\infty=0.
	\end{equation}
	Then claim \eqref{eq:tight_2} follows by the previous identity and \eqref{eq:tight_10}.
	
	Now, one can observe that \eqref{eq:tight_1} and \eqref{eq:tight_2} readily imply that $\nu_\infty=0$. Furthermore, combining Hardy's inequality 
	\[
		\kappa_s h_{N,s}\int_{\R^N}\frac{\xi_R^2 u_n^2}{\abs{x}^{2s}}\dx\leq \int_{\R^{N+1}_+}t^{1-2s}\abs{\nabla (\xi_R u_n)}^2\dxdt
	\]
	with \eqref{eq:tight_7} and \eqref{eq:tight_3}, we derive that
	\begin{equation*}
		\kappa_s h_{N,s}\gamma_\infty\leq \mu_\infty.
	\end{equation*}
	This, in combination with \eqref{eq:tight_11}, implies that
	\[
		\left(h_{N,s}-\sum_{i=1}^k\lambda_i\right)\gamma_\infty\leq 0
	\]
	which, in view of \Cref{thm:FMO1} (remember that we are assuming $\mu_{\LA}>0$) tells us that $\gamma_\infty=0$. Finally, again from \eqref{eq:tight_11} we obtain that $\mu_\infty=0$ and the proof is complete.

\end{proof}

We conclude by stating the main result of the present section, which ensure relative compactness of Palais-Smale sequences, up to assuming the energy level stays below a certain threshold.

\begin{lemma}\label{lemma:PS_c}
	Let $\{u_n\}_n\sub\Dext$ be a Palais-Smale sequence for $J$ at level $c\in \R$.	We have that, if
	\begin{equation}\label{eq:ps_hp1}
		c< \frac{s\kappa_s}{N}S_{\LA}^{1-\frac{N}{2s}}\min \left\{S,S_{\lambda_1},\dots,S_{\lambda_k},S_{\sigma_{\bm{\lambda}}}\right\}^{\frac{N}{2s}},
	\end{equation}
	then $\{u_n\}_n$ has a convergent subsequence.
\end{lemma}
\begin{proof}
	First of all we observe that, in view of \Cref{lemma:bounded}, the sequence $\{u_n\}_n$ is bounded; hence, up to a subsequence (still labeled as $\{u_n\}_n$), we have that
	\begin{align*}
		& u_n\weak u &&\text{weakly in }\Dext,~\text{in }L^{2\gamma}_{\textup{loc}}(\R^{N+1}_+,t^{1-2s})~\text{and in }L^{2^*_s}(\R^N) \\
		& u_n\to u &&\text{strongly in }L^p_{\textup{loc}}(\R^{N+1}_+,t^{1-2s}) ~\text{and in }L^q_{\textup{loc}}(\R^N)
	\end{align*}
	as $n\to\infty$, for all $p\in [1,2\gamma)$ and all $q\in [1,2^*_s)$, for some $u\in\Dext$. In addition, 
	\begin{align*}
		t^{1-2s}\abs{\nabla u_n}^2&\weak  \mu \quad\text{weakly in }\mathcal{M}(\overline{\R^{N+1}_+}), \\
		\abs{ u_n}^{2^*_s}&\weak \nu \quad\text{weakly in }\mathcal{M}(\R^N), \\
		\frac{ u_n^2}{\abs{x-a_i}^{2s}}&\weak \gamma_i\quad\text{weakly in }\mathcal{M}(\R^N),~\text{for all }i=1,\dots,k,
	\end{align*}
	as $n\to\infty$, for some $\mu\in\mathcal{M}(\overline{\R^{N+1}_+})$, $\nu\in\mathcal{M}(\R^N)$,  and $\{\gamma_i\}_{i=1,\dots,k}\sub\mathcal{M}(\R^N)$. Therefore, since the assumptions of \Cref{lemma:tight} are satisfied, the sequence $\{u_n\}_n$ is tight; hence, we are in position to use the concentration compactness principle, as stated in \Cref{thm:conc_comp}, from which we deduce that
	\begin{align*}
		&\mu\geq t^{1-2s}|\nabla u|^2\dxdt+\sum_{i=1}^k \mu_{a_i}\delta_{a_i}+\sum_{j\in \mathcal{J}}\mu_{x_j}\delta_{x_j}, \\
		&\nu =|u|^{2^*_s}\dx+\sum_{i=1}^k \nu_{a_i}\delta_{a_i}+\sum_{j\in \mathcal{J}}\nu_{x_j}\delta_{x_j}, \\
		&\gamma_i=\frac{u^2}{|x-a_i|^{2s}}\dx+\gamma_{a_i}\delta_{a_i},
	\end{align*}
	for some $\mu_{a_i},\nu_{a_i},\gamma_{a_i}\in\R$. In order to conclude the proof, we may prove that $\mu_{a_i}=\nu_{a_i}=\gamma_{a_i}=0$ for all $i=1,\dots,k$ and $\mu_{x_j}=\nu_{x_j}=0$ for all $j\in\mathcal{J}$. 
	
	We divide the rest of the proof into two other steps.
		
	\noindent\textbf{Step 1. } We claim that
	\begin{equation}\label{eq:ps_claim1}
		S_{\LA} \nu_{x_j}\geq S\nu_{x_j}^{\frac{2}{2^*_s}}\quad\text{for all }j\in\mathcal{J}.
	\end{equation}
	By assumption, we have that, for any $\epsilon>0$ and any $j\in\mathcal{J}$,
	\[
		\<J'(u_n),u_n\eta_\epsilon^{x_j}\>\to 0,\quad\text{as }n\to\infty,
	\]
	being $\{u_n\eta_\epsilon^{x_j}\}_n\sub\Dext$ a bounded sequence. Hence, for any $\epsilon>0$,
	\begin{multline}\label{eq:ps1}
		\lim_{n\to\infty}\Bigg[\int_{\R^{N+1}_+}t^{1-2s}\abs{\nabla u_n}^2|\eta_\epsilon^{x_j}|^2\dxdt+\int_{\R^{N+1}_+}t^{1-2s}u_n\nabla u_n\cdot\nabla \eta_\epsilon^{x_j}\dxdt \\
		-\kappa_s\sum_{i=1}^k\lambda_i\int_{\R^N}\frac{u_n^2\eta_\epsilon^{x_j}}{\abs{x-a_i}^{2s}}\dx-\kappa_s S_{\LA}\int_{\R^N}\abs{u_n}^{2^*_s}\eta_\epsilon^{x_j}\dx\Bigg]=0.
	\end{multline}
	In view of \Cref{thm:conc_comp} and the definition of $\eta_\epsilon^{x_j}$ we have that
	\begin{gather*}
		\lim_{\epsilon\to 0}\lim_{n\to\infty}\int_{\R^{N+1}_+}t^{1-2s}\abs{\nabla u_n}^2\eta_\epsilon^{x_j}\dxdt=\lim_{\epsilon\to 0}\int_{\R^{N+1}_+}\eta_\epsilon^{x_j}\d \mu\geq \mu_{x_j}, \\
		\lim_{\epsilon\to 0}\lim_{n\to\infty}\int_{\R^N}\abs{u_n}^{2^*_s}\eta_\epsilon^{x_j}\dx=\lim_{\epsilon\to 0}\int_{\R^N}\eta_\epsilon^{x_j}\d \nu=\nu_{x_j}
	\end{gather*}
	and, since $a_i\not\in \supp \eta_\epsilon^{x_j}$ for any $i=1,\dots,k$, for $\epsilon>0$ sufficiently small, also
	\[
		\lim_{\epsilon\to 0}\lim_{n\to\infty}\int_{\R^N}\frac{u_n^2\eta_\epsilon^{x_j}}{\abs{x-a_i}^{2s}}\dx=\lim_{\epsilon\to 0}\int_{\R^N}\frac{u^2 \eta_\epsilon^{x_j}}{\abs{x-a_i}^{2s}}\dx=0.
	\]
	Now, combining these facts and \eqref{eq:claim_0_2} with \eqref{eq:ps1} we obtain that
	\begin{equation}\label{eq:ps4}
		\mu_{x_j}\leq \kappa_s S_{\LA}\nu_{x_j}
	\end{equation}
	which, together with \eqref{eq:conc_comp_sobolev}, implies \eqref{eq:ps_claim1}.
	
	\noindent\textbf{Step 2.} We next claim that
	\begin{equation}\label{eq:ps_claim2}
		S_{\LA}\nu_{a_i}\geq S_{\lambda_i}\nu_{a_i}^{\frac{2}{2^*_s}}.
	\end{equation}
	The previous inequality comes as a consequence of the two following ones
	\begin{gather}
		\mu_{a_i}\geq \kappa_s(\lambda_i\gamma_{a_i}+S_{\lambda_i}\nu_{a_i}^{\frac{2}{2^*_s}}), \label{eq:ps2} \\
		\mu_{a_i} \leq \kappa_s(\lambda_i\gamma_{a_i}+S_{\LA}\nu_{a_i}), \label{eq:ps3}
	\end{gather}
	In order to prove \eqref{eq:ps2}, we first test the definition of $S_{\lambda_i,a_i}$ with $u_n\eta_\epsilon^{a_i}\in\mathcal{D}_t^{1,2}(\overline{\R^{N+1}_+})$, which gives
	\begin{multline*}
		\int_{\R^{N+1}_+}t^{1-2s}|\eta_\epsilon^{a_i}|^2|\nabla u_n|^2\dxdt+\int_{\R^{N+1}_+}t^{1-2s}|u_n|^2|\nabla \eta_\epsilon^{a_i}|^2\dxdt\\
		+2\int_{\R^{N+1}_+}t^{1-2s}\eta_\epsilon^{a_i}u_n\nabla \eta_\epsilon^{a_i}\cdot\nabla u_n\dxdt \geq \kappa_s\lambda_i\int_{\R^N}\frac{|\eta_\epsilon^{a_i}|^2|u_n|^2}{|x-a_i|^{2s}}\dx\\
		+\kappa_sS_{\lambda_i,a_i}\left(\int_{\R^N}|\eta_\epsilon^{a_i}u_n|^{2^*_s}\dx\right)^{\frac{2}{2^*_s}}.
	\end{multline*}
	Then, simply using the weak limit of $u_n$, the definition of $\eta_\epsilon^{a_i}$ and \eqref{eq:claim_0_1}, \eqref{eq:claim_0_2}, we get \eqref{eq:ps2}. Let us now prove \eqref{eq:ps3} by testing as follows
	\begin{equation*}
		\langle J'(u_n),u_n\eta_\epsilon^{a_i}\rangle\to 0\quad\text{as }n\to\infty.
	\end{equation*}
	If we expand this, we see that
	\begin{multline*}
		\lim_{n\to\infty}\int_{\R^{N+1}_+}t^{1-2s}|\nabla u_n|^2\eta_\epsilon^{a_i}\dxdt+\int_{\R^{N+1}_+}t^{1-2s}u_n\nabla u_n\cdot\nabla \eta_\epsilon^{a_i}\dxdt \\
		-\kappa_s\sum_{j=1}^k\lambda_j\int_{\R^N}\frac{u_n^2\eta_\epsilon^{a_i}}{|x-a_j|^{2s}}\dx-\kappa_s S_{\LA}\int_{\R^N}\eta_\epsilon^{a_i}|u_n|^{2^*_s}\dx=0.
	\end{multline*}
	Taking into account \eqref{eq:claim_0_2} and the fact that
	\begin{equation*}
		\lim_{\epsilon\to 0}\limsup_{n \to \infty}\int_{\R^N}\frac{u_n^2\eta_\epsilon^{a_i}}{|x-a_j|^{2s}}\dx=\begin{cases}
			0, &\text{if }i\neq j, \\
			\gamma_i,&\text{if }i=j,
		\end{cases}
	\end{equation*}
	together with the weak limit of $u_n$ in the various senses, we get \eqref{eq:ps3}. This also concludes the proof of \eqref{eq:ps_claim2}.
	
	\noindent\textbf{Conclusion.} By explicit computations (see \eqref{eq:J'u_u}), we have that
	\begin{equation*}
		c=\lim_{n\to\infty}J(u_n)=\lim_{n \to \infty}\left[J(u_n)-\frac{1}{2}\<J'(u_n),u_n\>\right]=\frac{s\kappa_s }{N}S_{\LA}\lim_{n \to \infty}\int_{\R^N}|u_n|^{2^*_s}\dx.
	\end{equation*}
	On the other hand, since $u_n\weak \nu$ as $n\to\infty$, thanks to the explicit form of $\nu$ and the tightness condition we have that
	\begin{equation}\label{eq:conc_1}
		c= \frac{s\kappa_s }{N}S_{\LA}\left(\int_{\R^N}|u|^{2^*_s}\dx+\sum_{j\in \mathcal{J}}\nu_{x_j} +\sum_{i=1}^k\nu_{a_i}\right).
	\end{equation}
	Let us now suppose by contradiction that $\nu_{x_j}>0$ for some $j\in\mathcal{J}$. By Step 1 \eqref{eq:ps_claim1}, we get that 
	\begin{equation*}
		\nu_{x_j}\geq \left(\frac{S}{S_{\LA}}\right)^{\frac{N}{2s}}
	\end{equation*}
	and then, plugging this into \eqref{eq:conc_1} and combining it with \eqref{eq:ps_hp1}, we find a contradiction. Hence, $\nu_{x_j}=0$ for all $j\in\mathcal{J}$. Reasoning analogously using Step 2 \eqref{eq:ps_claim2}, we get that $\nu_{a_i}=0$ for all $i=1,\dots,k$. In particular, we have that
	\begin{equation*}
		\int_{\R^N}|u_n|^{2^*_s}\dx\to \int_{\R^N}|u|^{2^*_s}\dx\quad\text{as }n\to\infty
	\end{equation*}
	and so $u_n\to u$ strongly in $L^{2^*_s}(\R^N)$, as $n\to\infty$. At this point, from \eqref{eq:ps4} we get that also $\mu_{x_j}=0$ for all $j\in\mathcal{J}$ while, from \eqref{eq:ps2}, \eqref{eq:ps3} we have that $\mu_{a_i}=\kappa_s\lambda_i\gamma_{a_i}$. By Hardy's inequality, we have that
	\begin{align*}
		\kappa_s h_{N,s}\int_{\R^N}\frac{|u_n|^2|\eta_\epsilon^{a_i}|^2}{|x-a_i|^{2s}}\dx &\leq \int_{\R^{N+1}_+}t^{1-2s}|\nabla (u_n\eta_\epsilon^{a_i})|^2\dxdt \\
		&=\int_{\R^{N+1}_+}t^{1-2s}|\eta_\epsilon^{a_i}|^2|\nabla u_n|^2\dxdt+\int_{\R^{N+1}_+}|u_n|^2|\nabla \eta_\epsilon^{a_i}|^2\dxdt \\
		&+2\int_{\R^{N+1}_+}t^{1-2s}u_n\eta_\epsilon^{a_i}\nabla u_n\cdot\nabla \eta_\epsilon^{a_i}\dxdt,
	\end{align*}
	from which, in view of Step 0, we derive that
	\begin{equation*}
		\kappa_s h_{N,s}\gamma_{a_i}\leq \mu_{a_i},
	\end{equation*}
	which then implies that $(h_{N,s}-\lambda_i)\gamma_{a_i}\leq 0$. On the other hand, by assumption $\lambda_i<h_{N,s}$, hence $\gamma_{a_i}=0$ and so also $\mu_{a_i}=0$. This implies that $u_n\to u$ strongly in $\mathcal{D}^{1,2}_t(\overline{\R^{N+1}_+})$, as $n\to\infty$, thus concluding the proof.
	
\end{proof}

\section{The case of a single pole}\label{sec:single}

A crucial role in our analysis is played by solutions to the following problem with a single pole (here assumed to be the origin):
\begin{equation}\label{eq:one_pole_pbm}
	\begin{bvp}
		(-\Delta)^s u-\frac{\lambda}{\abs{x}^{2s}}u&=u^{2^*_s-1}, &&\text{in }\R^N, \\
		u&> 0, && \text{in }\R^N,
	\end{bvp}
\end{equation}
to be intended in a weak sense, i.e. $u\in\Ds$, $u> 0$ a.e. and
\begin{equation*}
	\int_{\R^{N+1}_+}t^{1-2s}\nabla u\cdot\nabla\varphi\dxdt-\kappa_s\lambda\int_{\R^N}\frac{u\varphi}{|x|^{2s}}\dx=\kappa_s\int_{\R^N}|u|^{2^*_s-1}\varphi\dx\quad\text{for all }\varphi\in \mathcal{D}^{1,2}_t(\overline{\R^{N+1}_+}).
\end{equation*}
From \cite[Theorem 1.5]{Dipierro2016} we know that, for $\lambda\in [0,h_{N,s})$ this problem admits a solution, which we denote by $\phi^\lambda$. In particular, $\phi^\lambda$ is obtained as a suitable power of a minimizer of the problem
\begin{equation*}
	S_{\lambda}=\inf_{u\in\Ds\setminus\{0\}}\frac{\displaystyle\norm{u}_{\Ds}^2-\lambda\int_{\R^N}\frac{u^2}{\abs{x}^2}\dx}{\norm{u}_{L^{2^*_s}(\R^N)}^2}.
\end{equation*}
It is conjectured that the solution is unique, up to rescalings of the type
\begin{equation}\label{lambdamu}
		\phi^\lambda_\mu(x):=\mu^{-\frac{N-2s}{2}}\phi^\lambda\left(\frac{x}{\mu}\right),\qquad x\in\mathbb R^N\setminus\{0\}
\end{equation}
which are still solutions of \eqref{eq:one_pole_pbm} and minimizers of $S_\lambda$. To the best of our knowledge, this is still an open problem; however, knowing the existence of a solution  together with some decay properties (explained in the following), is sufficient to get our results. Moreover, for our subsequent needs, it is convenient to define the $L^{2^*_s}$-normalization of $\phi^\lambda_\mu$, i.e.
\begin{equation}\label{def:z_lambda}
	z^\lambda_\mu(x):=\frac{\phi^\lambda_\mu(x)}{\|\phi^\lambda_\mu\|_{L^{2^*_s}(\R^N)}}.
\end{equation}

\medskip

For any $\alpha\in\left( 0,\frac{N-2s}{2} \right)$ we define
\begin{equation}\label{eq:def_lambda_alpha}
	\Lambda(\alpha):=2^{2s}\frac{\Gamma\left(\frac{N+2s+2\alpha}{4}\right)\Gamma\left(\frac{N+2s-2\alpha}{4}\right)}{\Gamma\left(\frac{N-2s+2\alpha}{4}\right)\Gamma\left(\frac{N-2s-2\alpha}{4}\right)}.
\end{equation}
It is known (see e.g. \cite{Frank2008a} and  \cite[Proposition
2.3]{Fall2014}) that the map $\alpha\mapsto  \Lambda(\alpha)$ is
continuous and monotone decreasing. Moreover 
\begin{equation}\label{eq:8}
	\lim_{\alpha\to 0^+}\Lambda(\alpha)=h_{N,s},\qquad \lim_{\alpha\to \frac{N-2s}{2}}\Lambda(\alpha)=0.
\end{equation}
Therefore, the function
\[
\Lambda\colon \left[0,\frac{N-2s}{2}\right]\to [0,h_{N,s}]
\]
is invertible. Hereafter we define, for $\lambda\in[0,h_{N,s}]$,
\begin{equation}\label{eq:alpha}
	\alpha_\lambda:=\Lambda^{-1}(\lambda) \in\left[0,\frac{N-2s}{2}\right],
\end{equation}
which plays a crucial role in understanding the asymptotic behavior of solutions of \eqref{eq:one_pole_pbm}.

\medskip 

We start by stating the following result concerning the asymptotic behavior of solutions of \eqref{eq:one_pole_pbm} together with upper and lower bounds on the whole $\R^N$. Essentially, this comes as a consequence of \cite[Theorems 1.6 \& 1.7]{Dipierro2016} and \cite[Theorem 1.1]{Fall2014}.
\begin{lemma}\label{lemma:asy}
	Let $\phi\in\mathcal{D}^{s,2}(\R^N)$ be a weak solution of \eqref{eq:one_pole_pbm} with $\lambda>0$. Then $\phi$ is radial, radially decreasing and there exist $c_2\geq c_1>0$ such that
	\begin{equation}\label{twosided}
		\frac{c_1}{\left(\abs{x}^{1-\frac{2\alpha_\lambda}{N-2s}}(1+\abs{x}^{\frac{4\alpha_\lambda}{N-2s}})\right)^{\frac{N-2s}{2}}}\leq \phi(|x|)\leq \frac{c_2}{\left(\abs{x}^{1-\frac{2\alpha_\lambda}{N-2s}}(1+\abs{x}^{\frac{4\alpha_\lambda}{N-2s}})\right)^{\frac{N-2s}{2}}}
	\end{equation}
	for every $x\in \R^N\setminus\{0\}$. Moreover, there exists $\ell,\ell'\in[c_1,c_2]$ such that
	\begin{equation}\label{asy}
		\lim_{|x|\to+\infty}|x|^{\frac{N-2s}{2}+\alpha_\lambda}\,\phi(|x|)=\ell\qquad\text{and}\qquad\lim_{|x|\to 0}|x|^{\frac{N-2s}{2}-\alpha_\lambda}\,\phi(|x|)=\ell'	.
	\end{equation}
\end{lemma}
\begin{proof}
	The first part is precisely \cite[Theorems 1.6 \& 1.7]{Dipierro2016}, which we simply recalled. Hence, let us prove \eqref{asy}. Let us consider the eigenvalue problem
	\begin{equation*}
		\begin{bvp}
			-\mathrm{div}(\theta_{N+1}^{1-2s}\nabla_{\mathbb{S}^N}\psi)&=\mu\theta_{N+1}^{1-2s}\psi, &&\text{in }\mathbb{S}^N_+, \\
			\lim_{\theta_{N+1}\to 0^+}\theta_{N+1}^{1-2s}\nabla_{\mathbb{S}^N}\psi\cdot \bm{e}_{N+1}&=\kappa_s\lambda \psi, &&\text{on }\partial\mathbb{S}^N_+
		\end{bvp}
	\end{equation*}
	with $\mu\in\R$ being the eigenvalue and $\psi\in H^1(\mathbb{S}^N_+;\theta_{N+1}^{1-2s})\setminus\{0\}$ being a corresponding eigenfunction. By classical spectral theory (see e.g. \cite{Fall2014}), there exists a sequence of eigenvalues
	\begin{equation*}
		-\left(\frac{N-2s}{2}\right)^2<\mu_1(\lambda)<\mu_2(\lambda)\leq \cdots\leq \mu_k(\lambda)\leq \cdots\to+\infty.
	\end{equation*}
	In particular, we point out that the first eigenvalue is simple. From \cite[Theorem 1.1]{Fall2014}, we have that there exists $k\in\N$ and an eigenfunction $\psi\in H^1(\mathbb{S}^N_+;\theta_{N+1}^{1-2s})\setminus\{0\}$ corresponding to $\mu_k(\lambda)$ such that
	\begin{equation*}
		\lim_{\tau\to 0}\tau^{\frac{N-2s}{2}-\sqrt{\left(\frac{N-2s}{2}\right)^2+\mu_k(\lambda)}}\phi(\tau\theta')=\psi(\theta',0)\quad\text{for all }\theta\in\mathbb{S}^{N-1}=\partial \mathbb{S}^N_+.
	\end{equation*}
	On the other hand, since $\phi$ is radial and positive in $\R^N\setminus\{0\}$, then $\psi(\cdot,0)\equiv \ell'$, for some $\ell'>0$, i.e.
	\begin{equation}\label{eq:ff1}
		\lim_{\tau\to 0}\tau^{\frac{N-2s}{2}-\sqrt{\left(\frac{N-2s}{2}\right)^2+\mu_k(\lambda)}}\phi(\tau)=\ell'\quad\text{for all }\theta\in\mathbb{S}^{N-1}=\partial \mathbb{S}^N_+.
	\end{equation}
	We now claim that $k=1$. Indeed, by \eqref{twosided} we immediately deduce that
	\begin{equation}\label{eq:ff2}
		c_1\leq \liminf_{\tau \to 0}\tau^{\frac{N-2s}{2}-\alpha_\lambda}\phi(\tau)\leq \limsup_{\tau \to 0}\tau^{\frac{N-2s}{2}-\alpha_\lambda}\phi(\tau)\leq c_2.
	\end{equation}
	and this, in view of \eqref{eq:ff1}, forces
	\begin{equation*}
		\alpha_\lambda:=\sqrt{\left(\frac{N-2s}{2}\right)^2+\mu_1(\lambda)}=\sqrt{\left(\frac{N-2s}{2}\right)^2+\mu_k(\lambda)}
	\end{equation*}
	which holds if and only if $k=1$, by simplicity of the first eigenvalue. Hence, combining \eqref{eq:ff1} and \eqref{eq:ff2}, we obtain the second part of \eqref{asy}. In order to prove the first part, we first consider the $s$-fractional Kelvin transform
	\begin{equation*}
		K\colon \mathcal{D}^{s,2}(\R^N)\to \mathcal{D}^{s,2}(\R^N) 
	\end{equation*}
	defined, for $u\in\mathcal{D}^{s,2}(\R^N)$ as 
	\begin{equation*}
		Ku(x):=|x|^{2s-N}u\left(\frac{x}{|x|^2}\right).
	\end{equation*}
	From \cite[Corollary 3.2]{FallWeth} we deduce that $\tilde{\phi}:=K\phi$ solves \eqref{eq:one_pole_pbm}. We then reason analogously to the previous step, working with $\tilde{\phi}$ in place of $\phi$. From \cite[Theorem 1.1]{Fall2014} and the fact that $\tilde{\phi}$ is radial and positive, we get that there exists $k\in\N$  such that 
	\begin{equation}\label{eq:ff3}
		\lim_{\delta\to 0}\delta^{\frac{N-2s}{2}-\sqrt{\left(\frac{N-2s}{2}\right)^2+\mu_k(\lambda)}}\tilde{\phi}(\delta)=\ell,
	\end{equation}
	for some $\ell>0$. On the other hand, letting
	\begin{equation*}
		\phi(x)=|x|^{2s-N}\tilde{\phi}\left(\frac{x}{|x|^2}\right)
	\end{equation*}
	in \eqref{twosided}, and then $\delta=|x|^{-1}$, we deduce that
	\begin{equation*}
		\frac{c_1\delta^{2s-N}}{\left(\delta^{1-\frac{2\alpha_\lambda}{N-2s}}(1+\delta^{\frac{4\alpha_\lambda}{N-2s}})\right)^{\frac{N-2s}{2}}}\leq \tilde{\phi}(\delta)\leq \frac{c_2\delta^{2s-N}}{\left(\delta^{1-\frac{2\alpha_\lambda}{N-2s}}(1+\delta^{\frac{4\alpha_\lambda}{N-2s}})\right)^{\frac{N-2s}{2}}}.
	\end{equation*}
	Hence, we obtain that
	\begin{equation}\label{eq:ff4}
		c_1\leq \liminf_{\delta \to 0}\delta^{\frac{N-2s}{2}-\alpha_\lambda}\tilde{\phi}(\delta)\leq \limsup_{\delta \to 0}\delta^{\frac{N-2s}{2}-\alpha_\lambda}\tilde{\phi}(\delta)\leq c_2,
	\end{equation}
	which, combined with \eqref{eq:ff3}, implies that $k=1$ and that
	\begin{equation*}
		\lim_{\delta\to 0}\delta^{\frac{N-2s}{2}-\alpha_\lambda}\tilde{\phi}(\delta)=\ell\in[c_1,c_2].
	\end{equation*}
	Now, letting $\tilde{\phi}(\tau^{-1})=\tau^{N-2s}\phi(\tau)$, we prove the second part of \eqref{asy}, thus concluding the proof.
\end{proof}

Thanks to these properties, we obtain crucial asymptotic estimates for the Hardy norm of $(\phi^\lambda_\mu)^2$ as $\mu\to 0$.

\begin{lemma}\label{lemma:phi-est}
	Let $\lambda\in(0,h_{N,s})$. Let $\phi^\lambda\in\mathcal{D}^{s,2}(\R^N)$ be a weak solution of \eqref{eq:one_pole_pbm} and, for $\mu>0$, let $\phi^\lambda_\mu$ be as in \eqref{lambdamu}. Then, for any $\xi \in \R^N\setminus\{0\}$, we have that the following holds as $\mu\to 0$:
	\begin{align}
		&\int_{\mathbb R^N}\frac{\phi^\lambda_\mu(x)^2}{|x+\xi|^{2s}}\,dx=\ell^2\left(\int_{\mathbb R^N}\frac{|x|^{2s-N-2\alpha_\lambda}}{|x+\bm{e}_1|^{2s}}\,dx\right)\,\frac{\mu^{2\alpha_\lambda}}{|\xi|^{2\alpha_\lambda}}+o(\mu^{2\alpha_\lambda})&&\mbox{if $s>\alpha_\lambda$}, \label{case1} \\
		&\int_{\mathbb R^N}\frac{\phi^\lambda_\mu(x)^2}{|x+\xi|^{2s}}\,dx={\sigma_{N-1}\,\ell^2}\,\frac{\mu^{2s}|\log\mu|}{|\xi|^{2s}}+o(\mu^{2s}|\log\mu|)&&\mbox{if $s=\alpha_\lambda$},\label{case3}\\
		& \int_{\mathbb R^N}\frac{\phi^\lambda_\mu(x)^2}{|x+\xi|^{2s}}\,dx=\frac{\mu^{2s}}{|\xi|^{2s}}\int_{\mathbb R^N}\phi^\lambda(x)^2\,dx+o(\mu^{2s})&&\mbox{if $s<\alpha_\lambda$},\label{case2} 
	\end{align}
	where $\bm{e}_1:=(1,0,\dots,0)$ and
	\begin{equation*}
		\sigma_{N-1}:=\mathcal{H}^{N-1}(\partial B_1).
	\end{equation*}
\end{lemma}	
\begin{proof}
	First of all we observe that, thanks to \eqref{asy}, we deduce
	\begin{equation}\label{pointwise}
		\lim_{\mu\to 0}\mu^{2s-N-2\alpha_\lambda}\,\phi^\lambda\left(\frac{x}{\mu}\right)^2=\ell^2|x|^{2s-N-2\alpha_\lambda}\quad\text{for every $x\neq0$}.
	\end{equation}
	Suppose first that $s>\alpha_\lambda$. From \eqref{lambdamu} we have $\mu^{-2\alpha_\lambda}\, \phi^\lambda_\mu(x)^2=\mu^{2s-N-2\alpha_\lambda}\,\phi^\lambda(x/\mu)^2$, therefore thanks to \eqref{pointwise} we have the pointwise convergence
	\[
	\mu^{-2\alpha_\lambda}\,\phi^\lambda_\mu(x)^2\to \ell^2|x|^{2s-N-2\alpha_\lambda}
	\]
	as $\mu\to 0$, for every $x\in\mathbb R^N\setminus\{0\}$. Furthermore, \eqref{lambdamu} and \eqref{twosided} imply that for every $x\in\mathbb R^{N}\setminus\{0,-\xi\}$ and every $\mu>0$
	\[
	\frac{\mu^{-2\alpha_\lambda}\,\phi^\lambda_\mu(x)^2}{|x+\xi|^{2s}}\le \frac{c_2^2 |x|^{2s-N-2\alpha_\lambda}}{|x+\xi|^{2s}},
	\]
	where the right hand side belongs to $L^1(\mathbb R^N)$ since $N/2>s>\alpha_\lambda$. By dominated convergence we conclude that
	\[
	\mu^{-2\alpha_\lambda}\int_{\mathbb R^N}\frac{\phi^\lambda_\mu(x)^2}{|x+\xi|^{2s}}\,dx=\ell^2\int_{\mathbb R^N}\frac{|x|^{2s-N-2\alpha_\lambda}}{|x+\xi|^{2s}}\,dx+o(1)
	\]
	as $\mu\to0$.
	It is easy to check that the function $\Theta: \mathbb R^N \setminus \{0\} \to \mathbb R$ defined as
	\[
	\Theta(\xi):=\int_{\mathbb R^N}\frac{|x|^{2s-N-2\alpha_\lambda}}{|x+\xi|^{2s}}\,dx
	\]
	satisfies $\Theta(r\xi) = r^{-2\alpha_\lambda} \Theta(\xi)$ and is rotationally invariant. Therefore $\Theta(\xi)= |\xi|^{-2\alpha_\lambda} \Theta(\bm{e}_1)$ and \eqref{case1} follows.

	Suppose next that $\alpha_\lambda\ge s$. By a change of variables, we can write
	\begin{align}
		\mu^{-2s}\int_{\mathbb R^N}\frac{\phi^\lambda_\mu(x)^2}{|x+\xi|^{2s}}\,dx&=\int_{B_{\frac{|\xi|}{2\mu}}}\,\frac{\phi^\lambda(x)^2}{|\mu x+\xi|^{2s}}\,dx+\int_{\mathbb R^N\setminus B_{\frac{|\xi|}{2\mu}}}\frac{\phi^\lambda(x)^2}{|\mu x+\xi|^{2s}}\,dx \notag \\
		&=\int_{B_{\frac{|\xi|}{2\mu}}}\frac{\phi^\lambda(x)^2}{|\mu x+\xi|^{2s}}\,dx+R_\xi(\mu),\label{split}
	\end{align}
	with
	\begin{equation*}
		R_\xi(\mu):=\int_{\mathbb R^N\setminus B_{\frac{|\xi|}{2\mu}}}\frac{\phi^\lambda(x)^2}{|\mu x+\xi|^{2s}}\,dx.
	\end{equation*}
	We now want to bound the reminder term $R_\xi(\mu)$. We first let
	\begin{align*}
		&A_{\xi}:=\{x\in\mathbb R^N:|x-\xi|\ge |\xi|/2\} \\
		&F_{\xi}:=A_{\xi}\cap\{x\in\mathbb R^N: |x|<2|\xi|\} \\
		&G_{\xi}:=A_{\xi}\cap\{x\in\mathbb R^N: |x|\ge 2|\xi|\}.
	\end{align*}
	We then perform a change of variable $\mu x+\xi\mapsto x$, we use \eqref{twosided} and we notice that on $G_{\xi}$ there holds $|x-\xi|\ge|x|/2$, thus obtaining
	\begin{align}
		R_\xi(\mu)&=\mu^{-N}\int_{A_{\xi}}\frac{\phi^\lambda((x-\xi)/\mu)^2}{|x|^{2s}}\,dx \notag\\
		&\le c_2^2\,\mu^{2\alpha_\lambda-2s}\int_{A_{\xi}}\frac1{|x|^{2s}\,|x-\xi|^{N-2s+2\alpha_\lambda}}\,dx\notag\\
		&=c_2^2\,\mu^{2\alpha_\lambda-2s}\left(\int_{G_{\xi}}\frac1{|x|^{2s}\,|x-\xi|^{N-2s+2\alpha_\lambda}}\,dx+\int_{F_{\xi}}
		\frac1{|x|^{2s}\,|x-\xi|^{N-2s+2\alpha_\lambda}}\,dx\right)\notag\\
		&\le c_2^2\,\mu^{2\alpha_\lambda-2s}\left(\int_{G_{\xi}}\frac{1}{|x|^{2s}\,|x/2|^{N-2s+2\alpha_\lambda}}\,dx+\int_{F_{\xi}}
		\frac{1}{|x|^{2s}\,|\xi/2|^{N-2s+2\alpha_\lambda}}\,dx\right)\notag\\
		&\le C_{N,s,\lambda}\,\mu^{2\alpha_\lambda-2s}\left(\int_{\mathbb R^N\setminus B_{2|\xi|}}\frac{1}{|x|^{N+2\alpha_\lambda}}\,dx+\frac{1}{|\xi|^{N-2s+2\alpha_\lambda}}\int_{B_{2|\xi|}}
		\frac{1}{|x|^{2s}}\,dx\right),\label{longestimate1} 
	\end{align}
	where, in the last step we used the fact that $G_\xi\sub\R^N\setminus B_{2|\xi|}$ and that $F_\xi\sub B_{2|\xi|}$. In particular, from \eqref{longestimate1} we deduce that
	\begin{equation}\label{longestimate} 
		R_\xi(\mu)=\begin{cases}
			o(1),&\text{if }\alpha_\lambda>s, \\
			O(1),&\text{if }\alpha_\lambda=s
		\end{cases}\quad\text{as }\mu\to 0.
	\end{equation}
	
	Suppose now that $\alpha_\lambda>s$. Notice that in this case \eqref{twosided} implies $\phi^\lambda\in L^2(\mathbb R^N)$. The first integral on the right hand side of \eqref{split} converges as $\mu \to 0$, i.e.
	\begin{equation}\label{eq:lemma1}
		\int_{B_{\frac{|\xi|}{2\mu}}}\frac{\phi^\lambda(x)^2}{|\mu x+\xi|^{2s}}\,dx=	\int_{\R^N}\frac{\phi^\lambda(x)^2}{|\mu x+\xi|^{2s}}\,\chi_{B_{\frac{|\xi|}{2\mu}}}\,dx\to \frac{1}{|\xi|^{2s}}\int_{\mathbb R^N}\phi^\lambda(x)^2\,dx\quad\text{as }\mu \to 0.
	\end{equation}
	This follows by dominated convergence, since we have pointwise convergence and the domination
	\begin{equation*}
		\frac{\phi^\lambda(x)^2}{|\mu x+\xi|^{2s}}\,\chi_{B_{\frac{|\xi|}{2\mu}}}\leq 	\frac{\phi^\lambda(x)^2}{(|\xi|/2)^{2s}}\in L^1(\mathbb R^N)
	\end{equation*}
	Now, combining \eqref{split} with \eqref{eq:lemma1} and \eqref{longestimate}, we get that
	\[\mu^{-2s}\int_{\mathbb R^N}\frac{\phi^\lambda_\mu(x)^2}{|x+\xi|^{2s}}\,dx=|\xi|^{-2s}\int_{\mathbb R^N}\phi^\lambda(x)^2\,dx+o(1)\]
	as $\mu\to 0$, so that \eqref{case2} follows.
	
	Finally, assume that $\alpha_\lambda=s$, so that, by \eqref{twosided} we have $\phi^\lambda\in L^2_{\textup{loc}}(\mathbb R^N)$. In this case, we again start with \eqref{split}, writing
	\begin{multline}
		\mu^{-2s}\int_{\mathbb R^N}\frac{\phi^\lambda_\mu(x)^2}{|x+\xi|^{2s}}\,dx \\
		=\frac{1}{|\xi|^{2s}}\int_{B_{\frac{|\xi|}{2\mu}}}\phi^\lambda(x)^2\,dx+\int_{B_{\frac{|\xi|}{2\mu}}}\left(\frac1{|\mu x+\xi|^{2s}}-\frac1{|\xi|^{2s}}\right)\phi^\lambda(x)^2\,dx+R_\xi(\mu),\label{split1}
	\end{multline}
	and we analyze the asymptotic behavior of each term of the right-hand side as $\mu\to 0$. Starting with the first, by using radiality of $\phi^\lambda$ we that
	\[
	\int_{B_{\frac{|\xi|}{2\mu}}}\phi^\lambda(x)^2\,dx=\sigma_{N-1}\int_0^{\frac{|\xi|}{2\mu}}\phi^\lambda(r)^2 r^{N-1}\,dr
	\]
	and the estimate from below in \eqref{twosided} implies $\phi^\lambda\notin L^2(\mathbb R^N)$,  so that by De L'H\"opital rule and by \eqref{asy} we have
	\begin{equation}\label{limitlog}
		\lim_{\mu\to0} \frac{\int_{B_{\frac{|\xi|}{2\mu}}}\phi^\lambda(x)^2\,dx}{-\log\mu}=\sigma_{N-1}\lim_{\mu\to 0}\frac{\phi^\lambda\!\left(\tfrac{|\xi|}{2\mu}\right)^2\left(\tfrac{|\xi|}{2\mu}\right)^{N-1}\tfrac{|\xi|}{2\mu^2}}{\tfrac1\mu}=\sigma_{N-1}\,\ell^2.
	\end{equation}
	Let us now pass to the second term in the right-hand side of \eqref{split1}. By the Taylor's formula, if $y\geq 0$, then
	\begin{equation*}
		|(1+y)^s-1|\leq s\sup_{z\in(0,y)}|1+z|^{s-1}|y|\leq s|y|,
	\end{equation*}
	while, if $-1<-c\leq y<0$ we have
	\begin{equation*}
		|(1+y)^s-1|\leq s\sup_{z\in(y,0)}|1+z|^{s-1}|y|\leq s|1-c|^{s-1}|y|.
	\end{equation*}
	Summing up, we have that
	\begin{equation}\label{eq:taylor}
		|(1+y)^s-1|\leq C|y|\quad\text{for all }y\geq -c>-1,
	\end{equation}
	where $C=\max\{s,s|1-c|^{s-1}\}$. Keeping in mind the second term in the right-hand side of \eqref{split1}, our aim is to estimate
	\begin{equation*}
		\left||\mu x+\xi|^{2s}-|\xi|^{2s}\right|=|\xi|^{2s}\left|\left(1+\mu^2|\xi|^{-2}|x|^2+2\mu|\xi|^{-2}\,{x\cdot\xi}\right)^s-1\right|
	\end{equation*}
	and so we want to apply \eqref{eq:taylor} to $y=|\xi|^{-2}(\mu^2|x|^2+2\mu\,{x\cdot\xi})$. Since $x\in B_{|\xi|/(2\mu)}$, i.e. $\mu |x|\leq |\xi|/2$, we deduce that
	\begin{equation*}
		y\geq |\xi|^{-2}(\mu^2|x|^2-2\mu|x||\xi|)= |\xi|^{-2}(|\xi|-\mu|x|)^2-1\geq -\frac{3}{4}>-1.
	\end{equation*}
	Therefore we have
	\begin{align*}
		\left||\mu x+\xi|^{2s}-|\xi|^{2s}\right|&=|\xi|^{2s}\left|\left(1+\mu^2|\xi|^{-2}|x|^2+2\mu|\xi|^{-2}\,{x\cdot\xi}\right)^s-1\right| \\
		&\leq C (\mu^2|\xi|^{2s-2}|x|^2+\mu|\xi|^{2s-1}|x|),
	\end{align*}
	up to enlarging $C>0$ (still depending only on $s$). Until the end of the proof, $C>0$ may change from line to line, possibly depending on $N$ and $s$.
		Hence,
		\begin{align}
			\int_{B_{\frac{|\xi|}{2\mu}}}\left|\frac1{|\mu x+\xi|^{2s}}-\frac1{|\xi|^{2s}}\right|\phi^\lambda(x)^2\,dx &\le 
			\frac{C}{|\xi|^{2s}}\int_{B_{\frac{|\xi|}{2\mu}}}\frac{\mu^2|\xi|^{2s-2}|x|^2+\mu|\xi|^{2s-1}|x|}{|\mu x+\xi|^{2s}}\,\phi^\lambda(x)^2\,dx\notag\\
			&\le \frac{C}{|\xi|^{4s}}\int_{B_{\frac{|\xi|}{2\mu}}}\left({\mu^2|\xi|^{2s-2}|x|^2+\mu|\xi|^{2s-1}|x|}\right)\,\phi^\lambda(x)^2\,dx \notag\\
			&=\frac{C\mu^2}{|\xi|^{2s+2}}\int_{B_{\frac{|\xi|}{2\mu}}}|x|^2\phi^\lambda(x)^2\,dx+\frac{C\mu}{|\xi|^{2s+1}}\int_{B_{\frac{|\xi|}{2\mu}}}|x|\phi^\lambda(x)^2\,dx\label{Ks}
		\end{align}
		Since for $\alpha_\lambda=s$ \eqref{twosided} implies $\phi^\lambda(x)^2\le c_2^2\,|x|^{-N}$, by radiality we have
		\[\int_{B_{\frac{|\xi|}{2\mu}}}|x|^2\,\phi^\lambda(x)^2\,dx=\sigma_{N-1} \int_0^{\frac{|\xi|}{2\mu}}r^2\phi^\lambda(r)^2r^{N-1}\,dr\le C\int_0^{\frac{|\xi|}{2\mu}}r\,dr=\frac{C|\xi|^2}{\mu^2}
		\]
		and similarly
		\[
		\int_{B_{\frac{|\xi|}{2\mu}}}|x|\,\phi^\lambda(x)^2\,dx\le \frac{C|\xi|}{\mu}.
		\]
		Therefore, from \eqref{Ks} we deduce that for every $\mu>0$
		\begin{equation}\label{eq:int}
			\int_{B_{\frac{|\xi|}{2\mu}}}\left|\frac1{|\mu x+\xi|^{2s}}-\frac1{|\xi|^{2s}}\right|\phi^\lambda(x)^2\,dx\le \frac{C}{|\xi|^{2s}}.
		\end{equation}
			Now, combining \eqref{split1} with \eqref{limitlog}, \eqref{eq:int} and \eqref{longestimate}
			we get
			\[
			\mu^{-2s}\int_{\mathbb R^N}\frac{\phi^\lambda_\mu(x)^2}{|x+\xi|^{2s}}\,dx=\frac{\sigma_{N-1}\,\ell^2}{|\xi|^{2s}}\,|\log\mu|+o(|\log\mu|)\qquad\mbox{as $\mu\to0$}
			\]
			thus proving \eqref{case3}.
		\end{proof}

\section{Conclusion}\label{sec:conclusion}
In this section, we collect the proof of our main results, which are essentially based on \Cref{lemma:PS_c} and \Cref{lemma:phi-est}. Let us here describe the path we follow. Let $\{u_n\}_n\sub\Ds$ be a minimizing sequence for $S_{\LA}$. First, by direct computations, one can prove that
\begin{equation*}
	G(u_n)\to \frac{s  S_{\LA}}{N}\quad\text{as }n\to+\infty.
\end{equation*}
Second, by Ekeland Variational Principle, one can prove that it is not restrictive to assume that $\{u_n\}_n$ is a Palais-Smale sequence for $G$. This is the content of \Cref{lemma:ex_ps}. Then, in view of \Cref{lemma:PS_c}, we find that we can extract a convergent subsequence if 
\begin{equation}\label{eq:PS_conclusion}
			S_{\LA}< \min \left\{S,S_{\lambda_1},\dots,S_{\lambda_k},S_{\sigma_{\bm{\lambda}}}\right\}.
\end{equation}
Hence, our main result \Cref{thm:main} follows once we can prove that \eqref{eq:PS_conclusion} holds, and this is done thanks to the asymptotic estimates proved in \Cref{lemma:phi-est} and by \Cref{lemma:interaction_conditions} below. We start with the following result, which provides the existence of a Palais-Smale sequence for $G$. The proof is standard in nonlinear analysis (essentially, Ekeland Variational Principle), but we provide it for the sake of completeness.
\begin{lemma}\label{lemma:ex_ps}
	There exists a Palais-Smale $\{u_n\}_n\sub\mathcal{D}^{s,2}(\R^N)$ sequence for $G$ at level $c=s  S_{\LA}/N$, i.e. the following holds as $n\to\infty$:
	\begin{enumerate}
		\item $G(u_n)\to c$
		\item $G'(u_n)\to 0$ in $\Big(\mathcal{D}^{s,2}(\R^N)\Big)^*$.
	\end{enumerate}
\end{lemma}
\begin{proof}
	For the sake of simplicity, in this proof we use the following notation:
	\begin{equation*}
		E(u,v)=\norm{u}_{\mathcal{D}^{s,2}(\R^N)}^2-\int_{\R^N}H_{\LA}u^2\dx,\quad E(u):=E(u,u),\quad \Phi(u):=\frac{E(u)}{\norm{u}_{2^*_s}^2}.
	\end{equation*}
	Let $\{u_n\}_n$ be a minimizing sequence for 
	$$  S_{\LA}=\inf_{u\in\mathcal{D}^{s,2}(\R^N)}\Phi(u),$$
	that is
	\begin{equation*}
		S_{\LA}\leq \Phi(u_n)\leq  S_{\LA}+\frac{1}{n}.
	\end{equation*}
	By homogeneity, we can assume $\norm{u_n}_{L^{2^*_s}(\R^N)}=1$.	From Ekeland Variational Principle \cite{ekeland} we know that there exists another sequence $\{v_n\}_n\sub\mathcal{D}^{s,2}(\R^N)$ such that
	\begin{align}
		&\Phi(v_n)\leq\Phi(u_n),\notag \\
		&\norm{u_n-v_n}_{\mathcal{D}^{s,2}(\R^N)}\leq \frac{1}{\sqrt{n}}, \label{eq:ex_ps_1}\\
		&\Phi(v_n)\leq \Phi(u)+\frac{1}{\sqrt{n}}\norm{u-v_n}_{\mathcal{D}^{s,2}(\R^N)}\quad\text{for all }u\in \mathcal{D}^{s,2}(\R^N).\notag
	\end{align}
	By testing the last inequality with $u=v_n\pm tv$, with $v\in \mathcal{D}^{s,2}(\R^N)$ and $t>0$ and letting $t\to 0^+$, we get that	
	\begin{equation*}
		|\Phi'(v_n)[v]|\leq \frac{\norm{v}_{\mathcal{D}^{s,2}(\R^N)}}{\sqrt{n}}
	\end{equation*}
	which implies that $\Phi'(v_n)\to 0$ as $n\to\infty $ in $(\mathcal{D}^{s,2}(\R^N))^*$. In particular, $v_n$ is a Palais-Smale sequence for $\Phi$ at level $  S_{\LA}$. We also observe that, by triangular inequality, Sobolev inequality, \eqref{eq:ex_ps_1} and the the fact that $\norm{u_n}_{L^{2^*_s}(\R^N)}=1$, we have
	\begin{equation}\label{eq:ex_ps_2}
		\norm{v_n}_{L^{2^*_s}(\R^N)}=\norm{u_n-v_n}_{L^{2^*_s}(\R^N)}+\norm{u_n}_{L^{2^*_s}(\R^N)}\leq \frac{S}{\sqrt{n}}+1.
	\end{equation}
	We now let $w_n:=\norm{v_n}_{2^*_s}^{-1}\,v_n$ and we observe that $\{w_n\}_n$ is still a Palais-Smale sequence for $\Phi$ at level $  S_{\LA}$. Indeed, by homogeneity $\Phi(w_n)=\Phi(v_n)$, while, by explicit computations
	\begin{equation*}
		\Phi'(w_n)[v]=\norm{v_n}_{2^*_s}\Phi'(v_n)[v],
	\end{equation*}
	which, combined with \eqref{eq:ex_ps_2}, yields 
	\begin{equation}\label{eq:ex_ps_3}
		\Phi'(w_n)\to 0 \quad\text{as }n\to\infty~\text{in }\Big(\mathcal{D}^{s,2}(\R^N)\Big)^*.
	\end{equation} 
	We finally claim that $w_n$ is a Palais-Smale sequence for $G$ at level $s  S_{\LA}$. First, we have
	\begin{equation*}
		G(w_n)=\frac{1}{2}E(w_n)-\frac{  S_{\LA}}{2^*_s}=\frac{1}{2}\Phi(v_n)-\frac{  S_{\LA}}{2^*_s}\to \frac{s  S_{\LA}}{N}.
	\end{equation*}
	Second, by explicit computations, one can check that
	\begin{equation*}
		\Phi'(w_n)[v]=2E(w_n,v)-2E(w_n)\int_{\R^N}|w_n|^{2^*_s-2}w_n v\dx,
	\end{equation*}
	and that
	\begin{equation*}
		G'(w_n)[v]=\frac{1}{2}\Phi'(w_n)[v]+E(w_n)\int_{\R^N}|w_n|^{2^*_s-2}w_n v\dx-S_{\LA}\int_{\R^N}|w_n|^{2^*_s-2}w_n v\dx.
	\end{equation*}
	By the minimizing properties of $w_n$ and \eqref{eq:ex_ps_3} we have that
	\begin{equation*}
		G'(w_n)[v]=\left(\norm{v}_{\mathcal{D}^{s,2}(\R^N)}+\int_{\R^N}|w_n|^{2^*_s-2}w_n v\dx \right)o(1),\quad\text{as }n\to\infty.
	\end{equation*}
	Hence, we are left to prove that 
	\begin{equation*}
		\abs{\int_{\R^N}|w_n|^{2^*_s-2}w_n v\dx}\leq C\norm{v}_{\mathcal{D}^{s,2}(\R^N)}\quad\text{for all }n\in\N,
	\end{equation*}
	for some $C>0$ independent from $n$. In fact, by H\"older inequality
	\begin{align*}
		\abs{	\int_{\R^N}|w_n|^{2^*_s-2}w_n v\dx}\leq 	\int_{\R^N}|w_n|^{2^*_s-1}|v|\dx &\leq \norm{w_n}_{L^{2^*_s}(\R^N)}^{\frac{N+2s}{N-2s}}\norm{v}_{L^{2^*_s}(\R^N)} \\
		&\leq S^{-1/2} \norm{v}_{\mathcal{D}^{s,2}(\R^N)}.
	\end{align*}
	The proof is thereby complete, up to renaming $w_n$ as $u_n$.
\end{proof}

By suitably choosing a test function for $S_{\LA}$, in combination of \Cref{lemma:phi-est} , we obtain the following energy bound.

\begin{lemma}\label{lemma:interaction_conditions}
	Let $\bm{a}=(a_1,\dots,a_k)\in\R^{kN}$ and $\bm{\lambda}=(\lambda_1,\dots,\lambda_k)\in\R^k$ be such that $0<\lambda_i<h_{N,s}$ and
	\begin{align*}
		&\sum_{j\neq i}\frac{\lambda_j}{|a_j-a_i|^{2\alpha_{\lambda_i}}}>0, \quad\text{if }s>\alpha_{\lambda_i} \\
		&\sum_{j\neq i}\frac{\lambda_j}{|a_j-a_i|^{2s}}>0, \quad\text{if }s\leq \alpha_{\lambda_i}.
	\end{align*}
	Then $S_{\LA}<S_{\lambda_i}$.
\end{lemma}
\begin{proof}
	Let us test the minimization problem for $S_{\LA}$ with $z_\mu^{\lambda_i}(\cdot-a_i)$. We obtain
	\begin{align*}
		S_{\LA}&\leq \|z_\mu^{\lambda_i}(\cdot-a_i)\|_{\mathcal{D}^{s,2}(\R^N)}^2-\lambda_i\int_{\R^N}\frac{|z_\mu^{\lambda_i}(x-a_i)|^2}{|x-a_i|^{2s}}\dx \\
		&\quad-\sum_{\substack{j=1 \\ j\neq i}}^k\lambda_j\int_{\R^N}\frac{|z_\mu^{\lambda_i}(x-a_i)|^2}{|x-a_j|^{2s}}\dx \\
		&=\|z_\mu^{\lambda_i}\|_{\mathcal{D}^{s,2}(\R^N)}^2-\lambda_i\int_{\R^N}\frac{|z_\mu^{\lambda_i}|^2}{|x|^{2s}}\dx -\sum_{\substack{j=1 \\ j\neq i}}^k\lambda_j\int_{\R^N}\frac{|z_\mu^{\lambda_i}|^2}{|x-a_j+a_i|^{2s}}\dx \\
		&=S_{\lambda_i}-\sum_{\substack{j=1 \\ j\neq i}}^k\lambda_j\int_{\R^N}\frac{|z_\mu^{\lambda_i}|^2}{|x-a_j+a_i|^{2s}}\dx.
	\end{align*}
	We now observe that, by definition and scaling invariance,
	\begin{align*}
		\int_{\R^N}\frac{|z_\mu^{\lambda_i}|^2}{|x-a_j+a_i|^{2s}}\dx&=\frac{1}{\|\phi_\mu^\lambda\|_{L^{2^*_s}(\R^N)}^2}\int_{\R^N}\frac{|\phi_\mu^{\lambda_i}|^2}{|x-a_j+a_i|^{2s}}\dx \\
		&=\frac{1}{\|\phi^\lambda\|_{L^{2^*_s}(\R^N)}^2}\int_{\R^N}\frac{|\phi_\mu^{\lambda_i}|^2}{|x-a_j+a_i|^{2s}}\dx 
	\end{align*}
	so that by \Cref{lemma:phi-est}
	\begin{equation*}
		S_{\LA}\leq S_{\lambda_i}-C_{N,s,\lambda_i}\left\{\begin{aligned}
			&\mu^{2\alpha_{\lambda_i}}\Bigg(\sum_{j\neq i}\frac{\lambda_j}{|a_j-a_i|^{2\alpha_{\lambda_i}}}+o(1)\Bigg), &&\text{if }s>\alpha_{\lambda_i} \\
			&\mu^{2s}|\log\mu|\Bigg(\sum_{j\neq i}\frac{\lambda_j}{|a_j-a_i|^{2s}}+o(1)\Bigg), &&\text{if }s=\alpha_{\lambda_i} \\
			&\mu^{2s}\Bigg(\sum_{j\neq i}\frac{\lambda_j}{|a_j-a_i|^{2s}}+o(1)\Bigg), &&\text{if }s<\alpha_{\lambda_i }
		\end{aligned}\right.
	\end{equation*}
	as $\mu\to 0$, with $C_{N,s,\lambda_i}>0$. The proof is thereby complete.
\end{proof}

We are now ready to prove our main result, which concerns existence of nontrivial solutions.
\begin{proof}[\underline{Proof of \Cref{thm:main}}]
	First of all, one can easily observe that the map $\lambda\mapsto S_\lambda$ is monotone non-increasing for any $\lambda<h_{N,s}$.
	In view of this monotonicity property, assumptions \eqref{eq:main_order}, \eqref{eq:main_lambda_k} and \eqref{eq:main_sum} have the following consequences:
	\begin{align*}
		&\eqref{eq:main_order}\implies S_{\lambda_k}\leq  S_{\lambda_i}\quad\text{for all }i=1,\dots,k, \\
		&\eqref{eq:main_lambda_k}\implies S_{\lambda_k}\leq S_0=S, \\
		&\eqref{eq:main_sum}\implies S_{\lambda_k}\leq S_{\sigma_{\bm{\lambda}}}.
	\end{align*}
	On the other hand, in view of \Cref{lemma:interaction_conditions}, \eqref{eq:main_cases} implies that $S_{\LA}<S_{\lambda_k}$. Summing up, we have that
	\begin{equation}\label{eq:main_1}
		S_{\LA}<\min\{S,S_{\lambda_1},\dots,S_{\lambda_k},S_{\sigma_{\bm{\lambda}}}\}.
	\end{equation}
	Now, let $\{u_n\}_n$ be a Palais-Smale sequence for $G_{\LA}$ at level $c=s  S_{\LA}/N$, whose existence is ensured by \Cref{lemma:ex_ps}. In view of \eqref{eq:main_1} we have that
	\begin{equation*}
		c=\frac{s  S_{\LA}}{N}<\frac{s  S_{\LA}^{1-\frac{N}{2s}}}{N}\min\{S,S_{\lambda_1},\dots,S_{\lambda_k},S_{\sigma_{\bm{\lambda}}}\}^{\frac{N}{2s}}.
	\end{equation*}	
	Hence, thanks to \Cref{lemma:PS_c} we have that, up to a subsequence,
	\begin{equation*}
		u_n\to u_0\quad\text{strongly in }\mathcal{D}^{s,2}(\R^N)~\text{as }n\to\infty,
	\end{equation*}
	for some $u_0\in\mathcal{D}^{s,2}(\R^N)$ such that $G_{\LA}(u_0)=sS_{\LA}/N$ and $G_{\LA}'(u_0)=0$.
	Since $G_{\LA}'(u_0)=0$, we can test the differential with $u_0$ itself to find that 
	\begin{equation*}
		Q_{\LA}(u_0) = \int_{\R^N} |u_0|^{2^*_s}\dx.
	\end{equation*}
	Substituting this identity into the energy functional yields
	\begin{equation*}
		\frac{s S_{\LA}}{N} = G_{\LA}(u_0) = \frac{1}{2}Q_{\LA}(u_0) - \frac{1}{2^*_s}\int_{\R^N} |u_0|^{2^*_s}\dx = \frac{s}{N}\int_{\R^N} |u_0|^{2^*_s}\dx,
	\end{equation*}
	which implies that 
\begin{equation*}
		\int_{\R^N} |u_0|^{2^*_s}\dx = S_{\LA}^{N/2s}.
\end{equation*}
	Consequently, the Rayleigh quotient of $u_0$ is precisely
	\begin{equation*}
		\frac{Q_{\LA}(u_0)}{\displaystyle\left(\int_{\R^N}|u_0|^{2^*_s}\dx\right)^{\frac{2}{2^*_s}}} = \frac{S_{\LA}^{N/2s}}{\left(S_{\LA}^{N/2s}\right)^{\frac{N-2s}{N}}} = S_{\LA}.
	\end{equation*}
	Therefore, $u_0$ is a nontrivial minimizer for $S_{\LA}$. 	
	We now show that the minimizer can be chosen to be strictly positive. Since there holds $\| |u| \|_{\Ds} \leq \|u\|_{\Ds}$, we have $Q_{\LA}(|u_0|) \leq Q_{\LA}(u_0)$ while the $L^{2^*_s}$-norm remains unchanged. Thus, $|u_0|$ is also a minimizer, and without loss of generality, we can assume $u_0 \geq 0$ a.e. in $\R^N$. By multiplying $u_0$ by the scaling factor $S_{\LA}^{\frac{1}{2-2^*_s}}$, we obtain a weak solution $u$ to \eqref{eq:problem}.
	Finally, we establish the regularity and strict positivity of the solution.
	For any compact set $K \subset \R^N \setminus \{a_1, \dots, a_k\}$, the singular potential $H_{\LA}$ is bounded and smooth. Hence, \cite[Proposition 3.2 \& Lemma 3.3]{Fall2014} yield that $u$ and $\nabla u$ are locally H\"older continuous in $\R^N\setminus\{a_1,\dots,a_k\}$, which proves the first claim. Moreover, since $u \geq 0$ in $\R^N$ and $u \not\equiv 0$, the Hopf Lemma as in \cite[Proposition 4.11]{Cabre2014} ensures that $u > 0$ in $\R^N \setminus \{a_1, \dots, a_k\}$. The proof is thereby complete.
\end{proof}

We now prove the first result in the direction of non-existence of solutions.
\begin{proof}[\underline{Proof of \Cref{lemma:Q_nonnegative}}]
	The proof is standard and essentially follows from the discrete Picone inequality
	\begin{equation}\label{eq:picone}
		(a-b)\left(\frac{c^2}{a}-\frac{d^2}{b}\right)\leq |c-d|^2\quad\text{for all }a,b> 0~\text{and all }c,d\in\R,
	\end{equation}
	see e.g. \cite{Frank2008,brasco}. Let $v\in C_c^\infty(\R^N\setminus\{a_1,\dots,a_k\})$ and let us test \eqref{eq:weak} with $v^2/(u+\varepsilon)$, for $\varepsilon>0$. We have
	\begin{multline*}
		\int_{\R^{2N}}(u(x)-u(y))\left(\frac{v^2(x)}{u(x)+\varepsilon}-\frac{v^2(y)}{u(y)+\varepsilon}\right)|x-y|^{-(N+2s)}\dxdy \\
		-\int_{\R^N}H_{\LA}u\frac{v^2}{u+\varepsilon}\dx=\int_{\R^N}u^{2^*_s-1}\frac{v^2}{u+\varepsilon}\dx\geq 0.
	\end{multline*}
	In view of \eqref{eq:picone}, the first term can be estimated from above by $\norm{v}_{\Ds}^2$, while for the second one, we can pass to the limit as $\varepsilon\to 0$ by Lebesgue Dominated convergence, since
	\begin{equation*}
		\abs{H_{\LA}u\frac{v^2}{u+\varepsilon}}\leq |H_{\LA}|v^2\in L^1(\{v>0\}),
	\end{equation*}
	thus completing the proof.
\end{proof}

Below we provide the proof of the fact that if the coefficients $\lambda_i$ of their sum $\sum_{i=1}^k\lambda_i$ are higher then the Hardy constant, then \eqref{eq:problem} does not admit solutions.
\begin{proof}[\underline{Proof of \Cref{thm:nonex}}]
	Let $\lambda_i>h_{N,s}$. Then, by optimality of $h_{N,s}$ in the Hardy inequality, we have that there exists $u\in\Ds$ such that 
	\begin{equation*}
		K(u):=\norm{u}_{\Ds}^2-\lambda_i\int_{\R^N}\frac{u^2}{|x-a_i|^{2s}}\dx<0.
	\end{equation*}
	Now, for any $\rho>0$, we let 
	$$u_\rho(x):=\rho^{-\frac{N-2s}{2}}u\left(\frac{x-a_i}{\rho}+a_i\right).$$
	By explicit computations one gets that
	\begin{equation*}
		\norm{u_\rho}_{\Ds}=\norm{u}_{\Ds}\quad\text{and}\quad \int_{\R^N}\frac{u_\rho^2}{|x-a_j|^{2s}}=\int_{\R^N}\frac{u^2}{\abs{x-a_i+\frac{a_i-a_j}{\rho}}^{2s}},
	\end{equation*}
	from which we deduce that
	\begin{equation*}
		Q_{\LA}(u_\rho)=K(u)-\sum_{j\neq i}\lambda_j\int_{\R^N}\frac{u^2}{\abs{x-a_i+\frac{a_i-a_j}{\rho}}^{2s}}.
	\end{equation*}
	Now, in view of \cite[Lemma 8.1]{Felli-Roberto} we have that
	\begin{equation*}
		\int_{\R^N}\frac{u^2}{\abs{x-a_i+\frac{a_i-a_j}{\rho}}^{2s}}\to 0\quad\text{as }\rho \to 0,
	\end{equation*}
	and so we can take $\rho$ sufficiently small so that $Q_{\LA}(u_\rho)\leq \frac{K(u)}{2}<0$, thus proving the claim if $\lambda_i>h_{N,s}$.
	
	Let us now assume that $\sum_{i=1}^k \lambda_i>h_{N,s}$. Reasoning analogously to the previous step, by optimality of the Hardy inequality, we have that there exists $u\in\Ds$ such that
	\begin{equation*}
		K'(u)=\norm{u}_{\Ds}^2-\left(\sum_{i=1}^k \lambda_i\right)\int_{\R^N}\frac{u^2}{|x|^{2s}}<0
	\end{equation*}
	and, by defining $u_\rho(x):=\rho^{-\frac{N-2s}{2}}u(x/\rho)$, we have that
	\begin{equation*}
		\norm{u_\rho}_{\Ds}=\norm{u}_{\Ds}\quad\text{and}\quad \int_{\R^N}\frac{u_\rho^2}{|x-a_j|^{2s}}=\int_{\R^N}\frac{u^2}{\abs{x-\frac{a_j}{\rho}}^{2s}}.
	\end{equation*}
	Hence, since \cite[Lemma 8.1]{Felli-Roberto} provides
	\begin{equation*}
		\int_{\R^N}\frac{u^2}{\abs{x-\frac{a_j}{\rho}}^{2s}}\to \int_{\R^N}\frac{u^2}{\abs{x}^{2s}}\quad\text{as }\rho\to+\infty,
	\end{equation*}
	we get that
	\begin{equation*}
		Q_{\LA}(u_\rho)=K'(u)+o(1)\quad\text{as }\rho \to +\infty.
	\end{equation*}
	Therefore, by taking $\rho$ sufficiently large, we conclude the proof.
\end{proof}

Finally, we prove that, under suitable assumptions on the coefficients $\{\lambda_i\}_i$, there are no ground states (i.e. $S_{\LA}$ is not achieved), independently on the choice of the poles $\{a_i\}_i$.
\begin{proof}[\underline{Proof of \Cref{thm:non_achieved}}]
	The proof follows the lines of the proof of \cite[Theorem 1.3]{FT}, hence we only sketch it. First of all, let us assume \eqref{eq:non_ex_th1}. In this case, as in the proof of \cite[Theorem 1.3]{FT}, we get that, by definition $S_{\LA}\geq S_{\lambda_k}$ and, testing the definition of $S_{\LA}$ with $z^{\lambda_k}_\mu$ and using \Cref{lemma:phi-est}, we derive that $S_{\lambda_k}\leq S_{\LA}$. Hence, $S_{\LA}=S_{\lambda_k}$. Now, let us assume by contradiction that $S_{\LA}$ is achieved by some $u\in\Ds\setminus\{0\}$. This implies that
	\begin{equation*}
		S_{\LA}=\frac{\displaystyle \norm{u}_{\Ds}^2-\int_{\R^N}H_{\LA}u^2\dx}{\norm{u}_{L^{2^*_s}(\R^N)}^2}=S_{\sigma_{\bm{\lambda}}}\leq \frac{\displaystyle \norm{u}_{\Ds}^2-\lambda_k\int_{\R^N}\frac{u^2}{|x-a_k|^{2s}}\dx}{\norm{u}_{L^{2^*_s}(\R^N)}^2}
	\end{equation*}
	which yields
	\begin{equation*}
		\sum_{i=1}^{k-1}\lambda_i\int_{\R^N}\frac{u^2}{|x-a_i|^{2s}}\dx\geq 0
	\end{equation*}
	thus raising a contradiction.
	
	Let us assume \eqref{eq:non_ex_th2}. First, we claim that if
		\begin{equation*}
			\sigma_{\bm{\lambda}}\in[0,h_{N,s})\quad\text{and}\quad\lambda_i\in [0,h_{N,s})\quad\text{for all }i=1,\dots,k,
		\end{equation*}
		then
		\begin{equation}\label{eq:S_LA=S_sig}
			S_{\LA}= S_{\sigma_{\bm{\lambda}}},
		\end{equation}
		where we recall $\sigma_{\bm{\lambda}}$ is as in \eqref{eq:sigma_lambda}. As a simple corollary of \cite[Lemma 8.1]{Felli-Roberto}, we get that for any $\lambda\in [0,h_{N,s})$ and any $\xi\in\R^N$, there holds
		\begin{equation}\label{eq:z_mu_la}
			\lim_{\mu\to +\infty}\left[\int_{\R^N}\frac{|z^\lambda_\mu|^2}{|x+\xi|^{2s}}\dx-\int_{\R^N}\frac{|z^\lambda_\mu|^2}{|x|^{2s}}\dx\right]=0
		\end{equation}
		Let us test $S_{\LA}$ with $z_\mu^{\sigma_{\bm{\lambda}}}$ as in \eqref{def:z_lambda}, for $\mu>0$. In view of \eqref{eq:z_mu_la} we obtain that, as $\mu\to+\infty$,
		\begin{align*}
			S_{\LA}&\leq \int_{\R^{N+1}_+}t^{1-2s}|\nabla z_\mu^{\sigma_{\bm{\lambda}}}|^2\dxdt-\kappa_s\sum_{i=1}^k\lambda_i\int_{\R^N}\frac{|z_\mu^{\sigma_{\bm{\lambda}}}|^2}{|x-a_i|^{2s}}\dx \\
			&=\int_{\R^{N+1}_+}t^{1-2s}|\nabla z_\mu^{\sigma_{\bm{\lambda}}}|^2\dxdt-\kappa_s\left(\sum_{i=1}^k\lambda_i\right)\int_{\R^N}\frac{|z_\mu^{\sigma_{\bm{\lambda}}}|^2}{|x|^{2s}}\dx+o(1)=S_{\sigma_{\bm{\lambda}}}+o(1).
		\end{align*}
		By passing to the limit in the previous inequality, we get that $S_{\LA}\leq S_{\sigma_{\bm{\lambda}}}$. Now we prove the converse inequality. In order to prove this claim, for any $u\in\Ds$, $u\geq 0$, we consider its Schwarz symmetrization $u^*$. We know that
		\begin{equation}\label{eq:HL}
			\int_{\R^N}\frac{u^2}{|x-a_i|^{2s}}\dx\leq \int_{\R^N}(u^*)^2\Bigg(\left(\frac{1}{|x-a_i|}\right)^*\Bigg)^{2s}=\int_{\R^N}\frac{(u^*)^2}{|x|^{2s}}\dx
		\end{equation}
		and that $\norm{u^*}_{L^{2^*_s}(\R^N)}=\norm{u}_{L^{2^*_s}(\R^N)}$. Moreover, by the Pólya-Szegő inequality (see \cite[Theorem A.1]{Frank2008}) we know
		\begin{equation}\label{eq:polya}
			\norm{u^*}_{\Ds}\leq \norm{u}_{\Ds}.
		\end{equation}
		Therefore, we have
		\begin{equation}\label{eq:S_ineq}
			\frac{\displaystyle \norm{u}_{\Ds}^2-\int_{\R^N}H_{\LA}u^2\dx}{\norm{u}_{L^{2^*_s}(\R^N)}^2}\geq \frac{\displaystyle \norm{u^*}_{\Ds}^2-\sigma_{\bm{\lambda}}\int_{\R^N}\frac{(u^*)^2}{|x|^{2s}}\dx}{\norm{u^*}_{L^{2^*_s}(\R^N)}^2}\geq S_{\sigma_{\bm{\lambda}}}
		\end{equation}
		and, since the infimum defining $S_{\LA}$ can be taken, without loss of generality, among functions $u\geq 0$, we have that $S_{\LA}\geq S_{\sigma_{\bm{\lambda}}}$, thus proving \eqref{eq:S_LA=S_sig}. Let us now assume by contradiction that $S_{\LA}$ is achieved by some $u\in\Ds\setminus\{0\}$ and let us assume, without loss of generality, that $u\geq 0$ and that $\norm{u}_{L^{2^*_s}}=1$. From \eqref{eq:S_ineq} and \eqref{eq:S_LA=S_sig} we deduce that
		\begin{equation*}
			\norm{u}_{\Ds}^2-\int_{\R^N}H_{\LA}u^2\dx=\norm{u^*}_{\Ds}^2-\sigma_{\bm{\lambda}}\int_{\R^N}\frac{(u^*)^2}{|x|^{2s}}\dx.
		\end{equation*}
		From \eqref{eq:non_ex_th2}, \eqref{eq:polya} and \eqref{eq:HL} we deduce that
		\begin{equation*}
			0\leq \norm{u}_{\Ds}^2-\norm{u^*}_{\Ds}^2= \int_{\R^N}H_{\LA}u^2\dx-\sigma_{\bm{\lambda}}\int_{\R^N}\frac{(u^*)^2}{|x|^{2s}}\dx\leq 0,
		\end{equation*}
		which, in turn, gives that
		\begin{equation*}
			\norm{u}_{\Ds}^2=\norm{u^*}_{\Ds}^2.
		\end{equation*}
		Now, thanks to \cite[Theorem A.1]{Frank2008} we obtain that $u$ must be radially symmetric with respect to some $x_0\in\R^N$. Now we let $v=(S_{\LA})^{\frac{1}{2-2^*_s}}u$, which then solves \eqref{eq:problem}. The fact that $v$ is radially symmetric with respect to $x_0$ forces $H_{\LA}$ to be the same. Hence we reach a contradiction and we conclude the proof.

\end{proof}

\section*{Acknowledgements}
R. Ognibene is partially supported by the 2026 INdAM-GNAMPA project \emph{Asymptotic analysis of variational problems}, n. CUP\_E53C25002010001. E. Mainini is partially supported by the 2026 INdAM-GNAMPA project \emph{Modelli aggregazione-diffusione con mobilità non lineare: buona positura, comportamento asintotico}, n. CUP\_E53C25002010001.

\bibliography{biblio}
\bibliographystyle{aomalpha}

\end{document}